\documentclass[11pt, twoside, leqno]{article}

\usepackage{amssymb}
\usepackage{amsmath}
\usepackage{amsthm}
\usepackage{color}
\usepackage{mathrsfs}

\usepackage{indentfirst}

\usepackage{txfonts}

\allowdisplaybreaks

\def\red{\color{red}}

\def\rr{{\mathbb R}}
\def\rn{{\mathbb{R}^n}}

\def\zz{{\mathbb Z}}
\def\cc{{\mathbb C}}

\def\nn{{\mathbb N}}

\def\ca{{\mathcal A}}

\def\cf{{\mathcal F}}

\def\cm{{\mathcal M}}

\def\cs{{\mathcal S}}

\def\fz{\infty }

\def\lz{\lambda}

\def\lf{\left}
\def\r{\right}

\def\la{\langle}
\def\ra{\rangle}
\def\hs{\hspace{0.25cm}}
\def\ls{\lesssim}

\def\wh{\widehat}

\def\loc{{\mathop\mathrm{\,loc\,}}}
\def\supp{\mathop\mathrm{\,supp\,}}

\def\XXint#1#2#3{{\setbox0=\hbox{$#1{#2#3}{\int}$ }
\vcenter{\hbox{$#2#3$ }}\kern-.6\wd0}}

\DeclareMathOperator{\esssup}{ess\,sup}
\DeclareMathOperator{\essinf}{ess\,inf}

 \def\la{{\langle}}
 \def\ra{{\rangle}}
\def\({\left(}
\def \){ \right)}

\def\lz{{\lambda}}

 \def\BB{{\mathbb B}}

 \def\supp{\operatorname{supp}}

\newtheorem{theorem}{Theorem}[section]
\newtheorem{lemma}[theorem]{Lemma}

\newtheorem{assumption}[theorem]{Assumption}

\theoremstyle{definition}
\newtheorem{remark}[theorem]{Remark}
\newtheorem{definition}[theorem]{Definition}
\renewcommand{\appendix}{\par
   \setcounter{section}{0}%
   \setcounter{subsection}{0}%
   \setcounter{subsubsection}{0}%
   \gdef\thesection{\@Alph\c@section}%
   \gdef\thesubsection{\@Alph\c@section.\@arabic\c@subsection}%
   \gdef\theHsection{\@Alph\c@section.}%
   \gdef\theHsubsection{\@Alph\c@section.\@arabic\c@subsection}%
   \csname appendixmore\endcsname
 }

\numberwithin{equation}{section}

\begin{document}

\title{\bf\Large Littlewood-Paley Characterizations of
Hardy-type Spaces Associated with Ball Quasi-Banach Function Spaces\\
\vspace{0.5cm}
\small In Memory of Professor Carlos Berenstein
\footnotetext{\hspace{-0.35cm} 2010 {\it
Mathematics Subject Classification}. Primary 42B25;
Secondary 42B30, 42B35, 46E30.
\endgraf {\it Key words and phrases.} ball quasi-Banach function space, Hardy space,
Littlewood-Paley function, extrapolation, tent space, maximal function.
\endgraf This project is supported
by the National Natural Science Foundation of China (Grant Nos.
11971058, 11761131002 and 11671185).
Der-Chen Chang is partially supported by an NSF grant DMS-1408839 and a McDevitt Endowment
Fund at Georgetown University.}}
\author{Der-Chen Chang, Songbai Wang, Dachun Yang\,\footnote{Corresponding 
author/{\red November 9, 2019}/Final version.}
\ and Yangyang Zhang\,$^{\ast}$}
\date{}
\maketitle

\vspace{-0.7cm}

\begin{center}
\begin{minipage}{13cm}
{\small {\bf Abstract}\quad Let $X$ be a ball quasi-Banach function space on ${\mathbb R}^n$.
In this article,
assuming that the powered Hardy--Littlewood maximal operator
satisfies some Fefferman--Stein vector-valued maximal inequality on $X$ and is
bounded on the associated space,
the authors establish various Littlewood--Paley
function characterizations of the Hardy space $H_X({\mathbb R}^n)$
associated with $X$, under some weak assumptions on the Littlewood--Paley
functions. To this end, the authors also establish a useful estimate on
the change of angles in tent spaces associated with $X$.
All these results have wide applications. Particularly,
when $X:=M_r^p({\mathbb R}^n)$ (the Morrey space),
$X:=L^{\vec{p}}({\mathbb R}^n)$ (the mixed-norm Lebesgue space),
$X:=L^{p(\cdot)}({\mathbb R}^n)$ (the variable Lebesgue space),
$X:=L_\omega^p({\mathbb R}^n)$ (the weighted Lebesgue space)
and $X:=(E_\Phi^r)_t({\mathbb R}^n)$ (the Orlicz-slice space),
the Littlewood--Paley function characterizations of $H_X({\mathbb R}^n)$ obtained in this
article improve the existing results via weakening the assumptions on the Littlewood--Paley
functions and  widening
the range of $\lambda$ in
the Littlewood--Paley $g_\lambda^*$-function characterization of $H_X(\mathbb R^n)$.
}
\end{minipage}
\end{center}


\vspace{0.2cm}

\section{Introduction\label{s1}}
The real-variable theory of the classical Hardy space $H^p(\rn)$ with $p\in (0,1]$ was originally
initiated
by Stein and Weiss \cite{SW} and further developed by Fefferman and Stein \cite{FS}.
It is well known that the classical Hardy space $H^p(\rn)$ with $p\in (0,1]$ plays
a key role in harmonic analysis, partial differential equations
and other analysis subjects.
In particular, when $p\in(0,1]$,
$H^p(\rn)$ is a good substitute of the Lebesgue space $L^p(\rn)$
in the study on the boundedness of Calder\'on--Zygmund operators.
In recent decades, in order to meet the requirements arising in the study on the boundedness of operators,
partial differential equations and some other analysis subjects, various variants of Hardy spaces
have been introduced and
their real-variable theories have been well developed; these variants of Hardy spaces
were built on some elementary function spaces such as
weighted Lebesgue
spaces (see \cite{ST}), (weighted) Herz spaces (see, for instance,
\cite{CL,GC1,GCH,LY1,LY2}),  (weighted) Morrey spaces (see, for instance
\cite{JW,Sa,H}), Orlicz spaces (see, for instance, \cite{IV,Se,V,NS1,YY}),
Lorentz spaces (see, for instance, \cite{AST}), Musielak--Orlicz spaces
(see, for instance, \cite{K,YLK}) and variable function spaces (see, for instance, \cite{CUW,NS, YZN}).
Observe that these aforementioned elementary function spaces
are all included in a generalized framework called ball quasi-Banach function spaces
which were introduced, very recently, by Sawano et al. \cite{SHYY}.
Moreover, Sawano et al. \cite{SHYY}, Wang et al. \cite{WYY} and Zhang et al \cite{WYYZ,ZWYY}
established a unified real-variable theory
for Hardy spaces and weak Hardy spaces associated with ball quasi-Banach function spaces on ${\mathbb R}^n$
and gave some applications of these Hardy-type spaces to the boundedness of
Calder\'on--Zygmund operators and pseudo-differential operators.
More
function spaces based on ball quasi-Banach function spaces can be found in Sawano \cite{s}.

Recall that the original work of the Littlewood--Paley theory should be owned to Littlewood
and Paley \cite{LP}. Moreover, the Littlewood--Paley theory of Hardy spaces was further developed
by
Calder\'on \cite{C} and Fefferman and Stein \cite{FS}.
In recent decades, the Littlewood--Paley theory of
various variants of Hardy spaces has been well developed;
see, for instances, \cite{FoS,HYY,HLYY,lfy,LHY,lyy18,WYYZ,YYYZ,ZYYW,ZSY} and the related references.
Particularly, Folland and Stein \cite{FoS} obtained the Littlewood--Paley function
characterizations of Hardy spaces on homogeneous groups.
Observe that, in the case of $H^p(\rn)$ with $p\in(0,\infty)$, the best known range of
the parameter $\lambda$ in the $g_\lambda^\ast$-function characterization of Hardy spaces in \cite{FoS}
is $(\max\{1,2/p\},\infty)$ and
the function $\varphi$ appearing in the definitions of the Littlewood--Paley functions in \cite{FoS}
only need to
satisfy zero order vanishing moment.
However, compared with
the results in \cite{FoS} on the Littlewood--Paley function characterizations in the case of $H^p(\rn)$ with $p\in(0,\infty)$,
the range of $\lambda$ in the $g_\lambda^\ast$-function characterization appearing in \cite{HLYY,ZYYW,WYY}
does not coincide with the range in \cite{FoS}, namely, $\lambda\in(\max\{1,2/p\},\infty)$, and, in \cite{HYY,HLYY,lfy,LHY,lyy18,YYYZ}, the function $\varphi\in\cs(\rn)$ appearing in the
definitions of Littlewood--Paley functions is supposed to be
supported in the unit ball and have at least vanishing moments up to order
$\lfloor n(\frac1p-1)\rfloor$, which is much stronger than the corresponding assumptions on $\varphi$ appearing in the
definitions of Littlewood--Paley functions in \cite{FoS}. Here and thereafter,
the symbol $\lfloor s\rfloor$  for any $s\in\mathbb{R}$
denotes the largest integer not greater
than $s$.

Let $X$ be a ball quasi-Banach function space on $\rn$ introduced in \cite{SHYY}.
Assuming that the powered Hardy--Littlewood maximal operator
satisfies some Fefferman--Stein vector-valued maximal inequality on $X$ as well as it is
bounded on the associated space,
Sawano et al. \cite{SHYY}  established the Lusin area function characterization
of $H_X(\rn)$. Recently, Wang et al. \cite{WYY} obtained the Littlewood--Paley $g$ function
and the Littlewood--Paley $g_\lambda^*$-function characterizations of $H_X(\rn)$. We should point out
that, to characterize $H_X(\rn)$, Sawano et al. \cite{SHYY} and Wang et al. \cite{WYY}
required that the function $\varphi\in\cs(\rn)$ appearing in the definitions of Littlewood--Paley functions
satisfies $
\mathbf1_{B(\vec 0_n,4)\setminus B(\vec0_n,2)}\le\widehat\varphi\le
\mathbf1_{B(\vec 0_n,8)\setminus B(\vec0_n,1)}$.
Although, when $p\in(0,1]$ and $X:=L^p(\rn)$, the range of $\lambda$ in
the Littlewood--Paley $g_\lambda^*$-function
characterization of $H_X(\rn)$ obtained by Wang et al. \cite{WYY} coincides with
the best
known one in \cite{FoS}, namely, $\lambda\in(2/p,\infty)$, the range of $\lambda$
in \cite{WYY}
is not optimal even when $p\in(1,\infty)$ and $X:=L^p(\rn)$.
This point motivates us to optimize the range of the parameter $\lambda$
appearing in the Littlewood--Paley $g_\lambda^*$-function
characterization of $H_X(\rn)$ in \cite{WYY}.
Recently, when studying the Littlewood--Paley function characterizations of weak Hardy type spaces
associated with ball quasi-Banach function spaces,
Wang et al. \cite{WYYZ} found that it suffices to require that the function $\varphi\in\cs(\rn)$
satisfies that $\widehat\varphi(\vec0_n)=0$ and, for any $x\in\rn\setminus\{\vec 0_n\}$,
there exists a $t\in(0,\infty)$ such that $\widehat\varphi(tx)\not=0$. This motivates us to improve
the existing results of the Littlewood--Paley function characterizations of $H_X(\rn)$
in \cite{SHYY,WYY} by weakening the assumption on $\varphi$.
In this article,
we re-establish various Littlewood--Paley
function characterizations of $H_X(\rn)$ via weakening the assumptions on the Littlewood--Paley
functions and widening the range of $\lambda$ in the Littlewood--Paley $g_\lambda^*$-function
characterization of $H_X(\rn)$. To this end, we also establish a useful estimate on
the change of angles in tent spaces associated with $X$ (see Theorem \ref{pp} below).
Particularly, the assumptions on the Littlewood--Paley
functions in this article are much weaker than the corresponding assumptions
in \cite{HYY,HLYY,lfy,LHY,lyy18,WYY,SHYY,YYYZ}
(see Remark \ref{wa} below).
Besides, under these weaker assumptions on $\varphi$
in the definition of Littlewood--Paley functions,
the $g_\lambda^\ast$-function characterization of $H_X(\rn)$ obtained in this
article improves the existing results via widening the range of
$\lambda$
in \cite{WYY} (see Theorem \ref{Thgx} below).
We point out that the $\varphi$ appearing in the definition of the Littlewood--Paley functions
only need to
satisfy a zero order vanishing moment, which coincides with the corresponding assumptions in \cite{FoS}
and, when $X:=L^p(\rn)$ with $p\in(0,\infty)$,
the range of $\lambda$ in
the Littlewood--Paley $g_\lambda^*$-function
characterization of $H_X(\rn)$ in Theorem \ref{Thgx} below coincides with
the best known one in \cite{FoS}, namely, $(\max\{1,2/p\},\infty)$.
In \cite{WYY}, the estimate on
the change of angles in tent spaces associated with $X$
plays a key role in the proof of the Littlewood--Paley $g_\lambda^*$-function
characterization of $H_X(\rn)$.
However, to optimize the range of $\lambda$ in the Littlewood--Paley $g_\lambda^*$
function characterization of $H_X(\rn)$ in \cite{WYY}, the
estimate on the change of angles in tent spaces associated with $X$ in \cite{WYY} is no longer feasible.
To establish a more precise estimate on the change of angles (see Theorem \ref{pp} below),
instead of applying the atomic characterization of the tent space which was used in \cite{WYY},
we employ a method different from \cite{WYY}, namely,
we now use an extrapolation theorem over ball Banach function spaces (see Lemma \ref{thet} below)
which was proved in \cite[Lemma 7.34]{ZWYY},
and then establish a more refined estimate on the change of angles (see Theorem \ref{pp} below).
The assumptions in this estimate on the ball quasi-Banach function space $X$
are much weaker than the corresponding assumptions in \cite[Lemma 2.20]{WYY}.
All of these results have wide applications.
When $X:=M_r^p({\mathbb R}^n)$ (the Morrey space), $X:=L^{\vec{p}}(\rn)$ (the mixed-norm Lebesgue space),
$X:=L^{p(\cdot)}({\mathbb R}^n)$ (the variable Lebesgue space),
$X:=L_\omega^p({\mathbb R}^n)$ (the weighted Lebesgue space) and $X:=(E_\Phi^r)_t({\mathbb R}^n)$ (the Orlicz-slice space),
the Littlewood--Paley function characterizations of $H_X({\mathbb R}^n)$ obtained in this
article improve the existing results in \cite{HLYY,SHYY,WYY,YYYZ,ZYYW} via weakening
the assumptions on the Littlewood--Paley functions and widening the range of $\lambda$ in the Littlewood--Paley $g_\lambda^*$-function
characterization.

To be precise, this article is organized as follows.

In Section \ref{s2},
we recall some notions concerning the ball (quasi)-Banach function space $X$.
Then we state the assumptions of the Fefferman--Stein vector-valued
maximal inequality on $X$ (see Assumption \ref{as1}  below)
and the boundedness on the $r$-convexification of its associated space
for the Hardy--Littlewood maximal operator
(see Assumption \ref{as2} below). Finally, we recall
the extrapolation theorem over ball quasi-Banach function spaces proved in \cite{ZWYY} and
some notions
about the
Hardy space $H_X(\rn)$ introduced in \cite{SHYY}.

In Section \ref{s3}, via \cite[Proposition 3.2]{MPA} (see Lemma \ref{coa} below) and
the extrapolation theorem (see Lemma \ref{thet} below)
which was proved in \cite[Lemma 7.34]{ZWYY}, we establish an estimate on
the change of angles in tent spaces associated with a ball quasi-Banach function space $X$
(see Theorem \ref{pp} below), which
plays a key role in the proof of the Littlewood--Paley $g_\lambda^*$-function
characterization of $H_X(\rn)$ (see Theorem \ref{Thgx} below).

Section \ref{s4}
contains some square function characterizations
of $H_X(\rn)$, including its characterizations via the Lusin area function, the Littlewood--Paley
$g$-function and the Littlewood--Paley $g_\lambda^\ast$-function, respectively,
in Theorems \ref{Tharea}, \ref{Thgx} and \ref{Thgf} below.
We first prove Theorem \ref{Tharea}, the Lusin area function characterization of $H_X(\rn)$.
To this end, via borrowing some ideas from \cite{WYYZ}, we use
the atomic characterization of the tent space
associated to $X$ (see Lemma \ref{lt} below)
to decompose a distribution $f$ into a sequence of molecules; then,
applying some ideas used in the proof of \cite[Theorem 3.7]{WYYZ},
we prove Theorem \ref{Tharea}
under some even weaker assumptions on the Lusin area function.
After we obtain the Lusin area function characterization of $H_X(\rn)$,
using the estimate on the change of angles (see Theorem \ref{pp} below),
we establish the Littlewood-Paley $g_\lambda^*$-function characterization of $H_X(\rn)$,
namely, we prove Theorem \ref{Thgx}.
Finally, applying
an estimate initiated
by Ullrich \cite{U} and further improved by Wang et al. \cite{WYYZ} (see Lemma \ref{Le67} below), we
obtain the Littlewood--Paley $g$-function characterization of $H_X(\rn)$, namely,
we show Theorem \ref{Thgf}.

In Section \ref{s5}, we apply the above results, respectively, to the Morrey space,
the mixed-norm Lebesgue space, the variable Lebesgue space, the weighted Lebesgue space
and the Orlicz-slice space. Recall that, in these five examples, only variable Lebesgue spaces
are quasi-Banach function spaces and the others are only ball quasi-Banach function spaces.

Finally, we make some conventions on notation. Let $\nn:=\{1,2,\ldots\}$, $\zz_+:=\nn\cup\{0\}$
and $\zz_+^n:=(\zz_+)^n$.
We always denote by $C$ a \emph{positive constant} which is independent of the main parameters,
but it may vary from line to line. We also use $C_{(\alpha,\beta,\ldots)}$ to denote a positive constant depending
on the indicated parameters $\alpha,\beta,\ldots.$ The \emph{symbol} $f\lesssim g$ means that $f\le Cg$.
If $f\lesssim g$ and $g\lesssim f$, we then write $f\sim g$. We also use the following
convention: If $f\le Cg$ and $g=h$ or $g\le h$, we then write $f\ls g\sim h$
or $f\ls g\ls h$, \emph{rather than} $f\ls g=h$
or $f\ls g\le h$. The \emph{symbol} $\lfloor s\rfloor$  for any $s\in\mathbb{R}$
denotes the largest integer not greater
than $s$. We use $\vec0_n$ to denote the \emph{origin} of $\rn$ and let
$\mathbb{R}^{n+1}_+:=\rn\times(0,\infty)$.
If $E$ is a subset of $\rn$, we denote by $\mathbf{1}_E$ its
characteristic function and by $E^\complement$ the set $\rn\setminus E$.
For any
$\theta:=(\theta_1,\ldots,\theta_n)\in\zz_+^n$, let $|\theta|:=\theta_1+\cdots+\theta_n$. Furthermore,
for any ball $B$ in $\rn$ and $j\in\zz_+$, let $S_j(B):=(2^{j+1}B)\setminus(2^jB)$ with $j\in\nn$
and $S_0(B):=2B$. Finally, for any $q\in[1,\infty]$, we denote by $q'$ its \emph{conjugate exponent},
namely, $1/q+1/q'=1$.

\section{Preliminaries\label{s2}}

In this section, we first present some preliminary known facts on the ball quasi-Banach function space $X$
in \S\ref{s2.1}.
Then we state the assumptions of the Fefferman--Stein vector-valued
maximal inequality on $X$
and the boundedness on the $s$-convexification of $X$ for the Hardy--Littlewood maximal operator
in \S\ref{s2.2}. In \S\ref{s2.3}, we recall the extrapolation theorem associated with
the ball quasi-Banach function space $X$. Finally,
the notion of the Hardy type space $H_X({\mathbb R}^n)$ associated with $X$ was recalled in \S\ref{s2.4}.

\subsection{Ball quasi-Banach function spaces\label{s2.1}}
In this subsection, we recall some preliminary known facts on ball quasi-Banach function spaces
introduced in \cite{SHYY}.

Denote by the \emph{symbol} $\mathscr M(\rn)$ the set of
all measurable functions on $\rn$.
For any $x\in\rn$ and $r\in(0,\infty)$, let $B(x,r):=\{y\in\rn:\ |x-y|<r\}$ and
\begin{equation}\label{Eqball}
\BB:=\lf\{B(x,r):\ x\in\rn\quad\text{and}\quad r\in(0,\infty)\r\}.
\end{equation}

\begin{definition}\label{Debqfs}
A quasi-Banach space $X\subset\mathscr M(\rn)$ is called a \emph{ball quasi-Banach function space} if it satisfies
\begin{itemize}
\item[(i)] $\|f\|_X=0$ implies that $f=0$ almost everywhere;
\item[(ii)] $|g|\le |f|$ almost everywhere implies that $\|g\|_X\le\|f\|_X$;
\item[(iii)] $0\le f_m\uparrow f$ almost everywhere implies that $\|f_m\|_X\uparrow\|f\|_X$;
\item[(iv)] $B\in\BB$ implies that $\mathbf{1}_B\in X$, where $\BB$ is as in \eqref{Eqball}.
\end{itemize}

Moreover, a ball quasi-Banach function space $X$ is called a
\emph{ball Banach function space} if the norm of $X$
satisfies the triangle inequality: for any $f,\ g\in X$,
\begin{equation}\label{eq22x}
\|f+g\|_X\le \|f\|_X+\|g\|_X
\end{equation}
and, for any $B\in \BB$, there exists a positive constant $C_{(B)}$, depending on $B$, such that, for any $f\in X$,
\begin{equation}\label{eq2.3}
\int_B|f(x)|\,dx\le C_{(B)}\|f\|_X.
\end{equation}
For any ball Banach function space $X$, the \emph{associate space} (also called the
\emph{K\"othe dual}) $X'$ is defined by setting
\begin{equation}\label{asso}
X':=\lf\{f\in\mathscr M(\rn):\ \|f\|_{X'}:=\sup\lf\{\|fg\|_{L^1(\rn)}:\ g\in X,\ \|g\|_X=1\r\}<\infty\r\},
\end{equation}
where $\|\cdot\|_{X'}$ is called the \emph{associate norm} of $\|\cdot\|_X$
(see, for instance, \cite[Chapter 1, Definitions 2.1 and 2.3]{BS}).
\end{definition}

\begin{remark}\label{bbf}
\begin{itemize}
\item[(i)] By \cite[Proposition 2.3]{SHYY}, we know that, if $X$ is a ball Banach function space,
then its associate space $X'$ is also a ball Banach function space.
\item[(ii)] Recall that a quasi-Banach space $X\subset\mathscr M(\rn)$ is called a \emph{quasi-Banach function space} if it
is a ball quasi-Banach function space and it satisfies Definition \ref{Debqfs}(iv) with ball
replaced by any measurable set of finite measure (see, for instance, \cite[Chapter 1, Definitions 1.1 and 1.3]{BS}).
It is easy to see that every quasi-Banach function space is a ball quasi-Banach function space.
As was mentioned in \cite[p.\,9]{SHYY} and \cite[Section 5]{WYYZ}, the family of ball Banach function spaces includes Morrey spaces, mixed-norm Lebesgue spaces,  variable Lebesgue spaces,  weighted Lebesgue spaces
and Orlicz-slice space,
which are not necessary to be Banach function spaces.
\end{itemize}
\end{remark}

The following lemma is just \cite[Lemma 2.6]{ZWYY}.
\begin{lemma}\label{Lesdual}
Every ball Banach function space $X$ coincides with its second associate space $X''$.
In other words, a function $f$ belongs to $X$ if and only if it belongs to $X''$ and,
in that case,
$$
\|f\|_X=\|f\|_{X''}.
$$
\end{lemma}

We still need to recall the notion of the convexity of ball quasi-Banach spaces,
which is a part of \cite[Definition 2.6]{SHYY}.
\begin{definition}\label{Debf}
Let $X$ be a ball quasi-Banach function space and $p\in(0,\infty)$.
The $p$-\emph{convexification} $X^p$ of $X$ is defined by setting $X^p:=\{f\in\mathscr M(\rn):\ |f|^p\in X\}$
equipped with the quasi-norm $\|f\|_{X^p}:=\||f|^p\|_X^{1/p}$.
\end{definition}

\begin{lemma}\label{Re2.6}
Let $X$ be a ball Banach function space and $p\in[1,\infty)$. Then $X^p$ is a ball Banach function space.
\end{lemma}
\begin{proof}
Let $p\in[1,\infty)$. From the fact that $X$ is a ball Banach function space and the definition of
$X^p$, it easily follows that $X^p$ is a ball quasi-Banach function space.
Thus, to prove that $X^p$ is a ball Banach function space,
it suffices to show that, for any $f,\ g\in X^p$,
\begin{equation}\label{q2.3}
\lf\|f+g\r\|_{X^p}\le\|f\|_{X^p}+\|g\|_{X^p}
\end{equation}
and, for any $B\in \BB$, there exists a positive constant $C_{(B)}$, depending on $B$, such that,
for any $f\in X$,
\begin{equation}\label{q2.4}
\int_B|f(x)|\,dx\le C_{(B)}\|f\|_X.
\end{equation}
We first prove \eqref{q2.3}.
By Definition \ref{Debf}, Lemma \ref{Lesdual}, \eqref{asso} and the Minkowski inequality, we conclude
that
\begin{align*}
\lf\|f+g\r\|_{X^p}&=\lf\||f+g|^p\r\|_X^{1/p}=\lf\||f+g|^p\r\|_{X''}^{1/p}\\
&=\lf[\sup\lf\{\lf\|(f+g)^ph\r\|_{L^1(\rn)}:\ h\in X',\ \|h\|_{X'}=1\r\}\r]^\frac1p\\
&=\sup\lf\{\lf\|(f+g)h^\frac1p\r\|_{L^p(\rn)}:\ h\in X',\ \|h\|_{X'}=1\r\}\\
&\le\sup\lf\{\lf\|fh^\frac1p\r\|_{L^p(\rn)}+\lf\|gh^\frac1p\r\|_{L^p(\rn)}:\ h\in X',\ \|h\|_{X'}=1\r\}\\
&\le\lf[\sup\lf\{\lf\||f|^ph\r\|_{L^1(\rn)}:\ h\in X',\ \|h\|_{X'}=1\r\}\r]^\frac1p\\
&\quad+
\lf[\sup\lf\{\lf\||g|^ph\r\|_{L^1(\rn)}:\ h\in X',\ \|h\|_{X'}=1\r\}\r]^\frac1p\\
&=\lf\||f|^p\r\|_{X}^\frac1p+\lf\||g|^p\r\|_{X}^\frac1p=\|f\|_{X^p}+\|g\|_{X^p},
\end{align*}
which implies that \eqref{q2.3} holds true.

Now we show \eqref{q2.4}.
From the H\"older inequality and \eqref{eq2.3}, we deduce that, for any $B\in \BB$ and $f\in X^p$,
\begin{equation*}
\int_B|f(x)|\,dx\le \lf[\int_B|f(x)|^p\,dx\r]^{1/p}|B|^{1-1/p}
\le C_{(B)}\lf\||f|^p\r\|_{X}^{1/p}=C_{(B)}\lf\|f\r\|_{X^p},
\end{equation*}
where the positive constant $C_{(B)}$ is independent of $f$ but depending on $B$.
Thus, \eqref{q2.4} holds true, which, combined with \eqref{q2.3}, then completes
the proof of Lemma \ref{Re2.6}.
\end{proof}

\subsection{Assumptions on the Hardy--Littlewood maximal operator\label{s2.2}}

Denote by the \emph{symbol $L_{\loc}^1(\rn)$} the set of all locally integrable functions on $\rn$.
The \emph{Hardy--Littlewood maximal operator} $\cm$
is defined by setting, for any $f\in L_{\loc}^1(\rn)$ and $x\in\rn$,
\begin{equation}\label{mm}
\cm(f)(x):=\sup_{B\ni x}\frac1{|B|}\int_B|f(y)|\,dy,
\end{equation}
where the supremum is taken over all balls $B\in\BB$ containing $x$.

For any $\theta\in(0,\infty)$, the \emph{powered Hardy--Littlewood
maximal operator} $\cm^{(\theta)}$ is defined by setting,
for any $f\in L_{\loc}^1(\rn)$ and $x\in\rn$,
\begin{equation}\label{mmx}
\cm^{(\theta)}(f)(x):=\lf\{\cm\lf(|f|^\theta\r)(x)\r\}^{1/\theta}.
\end{equation}

The approach used in this article heavily depends on the following assumptions
on the boundedness of the Hardy--Littlewood maximal function on $X$, which is just \cite[(2.8)]{SHYY}.

\begin{assumption}\label{as1}
Let $X$ be a ball quasi-Banach function space. For some $\theta,\ s\in(0,1]$
and $\theta<s$, there exists a positive
constant $C$ such that, for any $\{f_j\}_{j=1}^\infty\subset\mathscr M(\rn)$,
\begin{equation}\label{as1-1}
\lf\|\lf\{\sum_{j=1}^\infty\lf[\cm^{(\theta)}(f_j)\r]^s\r\}^\frac1s\r\|_X\le C\lf\|\lf\{\sum_{j=1}^\infty|f_j|^s\r\}^\frac1s\r\|_X.
\end{equation}
\end{assumption}
\begin{assumption}\label{as2}
Let $X$ be a ball quasi-Banach function space.
Assume that there exist $s\in(0,\infty)$ and $q\in(s,\infty]$ such that $X^{1/s}$ is a ball Banach function space and,
for any $f\in (X^{1/s})'$,
\begin{equation}\label{as1-2}
\lf\|\cm^{((q/s)')}(f)\r\|_{(X^{1/s})'}\le C\lf\|f\r\|_{(X^{1/s})'},
\end{equation}
where the positive constant $C$ is independent of $f$.
\end{assumption}

\begin{lemma}\label{Le2.7}
Let $X$ be a ball Banach function space and $p\in[1,\infty)$. If $\cm$ is bounded on $X$,
then $\cm$ is also bounded on $X^p$.
\end{lemma}
\begin{proof}
By the H\"older inequality, we know that, for any $p\in[1,\infty)$
and any locally integrable function $f$,
$[\cm(f)]^p\le \cm(|f|^p)$.
From this and the fact that $\cm$ is bounded on $X$, we deduce that
\begin{equation}\label{Eq2.8}
\lf\|\cm(f)\r\|_{X^p}=\lf\|[\cm(f)]^p\r\|_X^\frac1p
\le\lf\|\cm(|f|^p)\r\|_X^\frac1p\lesssim\lf\||f|^p\r\|_X^\frac1p\sim\|f\|_{X^p}.
\end{equation}
This finishes the proof of Lemma \ref{Le2.7}.
\end{proof}
\begin{lemma}\label{Le2.9}
Let $X$ be a ball quasi-Banach function space.
Assume that there exists a $\theta\in(1,\infty)$ such that,
for any $f\in \mathscr M(\rn)$,
\begin{equation*}
\lf\|\cm^{(\theta)}(f)\r\|_{X}\le C\lf\|f\r\|_{X},
\end{equation*}
where the positive constant $C$ is independent of $f$.
Then
$\cm$ is bounded on $X$, namely, there exists a positive constant $C$ such
that, for any $f\in X$, $\|\cm(f)\|_{X}\le C\|f\|_{X}$.
\end{lemma}
\begin{proof}
From Definition \ref{Debf}, \eqref{mmx} and \eqref{as1-2},
it follows that, for any $f\in \mathscr M(\rn)$,
$$
\lf\|\cm(f)\r\|_{X^{1/\theta}}=\lf\|\lf[\cm(f)\r]^{1/\theta}\r\|_{X}^{\theta}
=\lf\|\cm^{(\theta)}(f^{1/\theta})\r\|_{X}^{\theta}
\lesssim\lf\|f\r\|_{X^{1/\theta}},
$$
which, together with Lemma \ref{Le2.7} and $\theta\in(1,\infty)$, further implies
$$
\lf\|\cm(f)\r\|_{X}\le C\lf\|f\r\|_{X}.
$$
This finishes the proof of Lemma \ref{Le2.9}.
\end{proof}

\subsection{Extrapolation theorem on ball Banach function spaces\label{s2.3}}

Now, we recall the notions of Muckenhoupt weights $A_p(\rn)$ (see, for instance, \cite{G1}).

\begin{definition}\label{weight}
An \emph{$A_p(\rn)$-weight} $\omega$, with $p\in[1,\infty)$, is a
locally integrable and nonnegative function on $\rn$ satisfying that,
when $p\in(1,\infty)$,
\begin{equation*}
\sup_{B\in\BB}\lf[\frac1{|B|}\int_B\omega(x)\,dx\r]\lf[\frac1{|B|}
\int_B\lf\{\omega(x)\r\}^{\frac1{1-p}}\,dx\r]^{p-1}<\infty
\end{equation*}
and, when $p=1$,
$$
\sup_{B\in\BB}\frac1{|B|}\int_B\omega(x)\,dx\lf[\lf\|\omega^{-1}\r\|_{L^\infty(B)}\r]<\infty,
$$
where $\BB$ is as in \eqref{Eqball}. Define $A_\infty(\rn):=\bigcup_{p\in[1,\infty)}A_p(\rn)$.
\end{definition}

\begin{definition}\label{wk}
Let $p\in(0,\infty)$ and $\omega\in A_\infty(\rn)$.
The \emph{weighted Lebesgue space $L_\omega^p(\rn)$} is defined
to be the set of all measurable functions $f$ such that
$$
\|f\|_{L^p_\omega(\rn)}:=\lf[\int_\rn|f(x)|^p\omega(x)\,dx\r]^\frac1p<\infty.
$$
\end{definition}

The following extrapolation theorem is just \cite[Lemma 7.34]{ZWYY},
which is a slight variant of a special case of \cite[Theorem 4.6]{CMP} via
replacing Banach function spaces by ball Banach function spaces.

\begin{lemma}\label{thet}
Let $X$ be a ball Banach function space and $p_0\in(0,\infty)$.
Let $\mathcal{F}$ be the set of all
pairs of nonnegative measurable
functions $(F,G)$ such that, for any given $\omega\in A_1(\rn)$,
$$
\int_{\rn}[F(x)]^{p_0}\omega(x)\,dx\leq C_{(p_0,[\omega]_{A_1(\rn)})}
\int_{\rn}[G(x)]^{p_0}\omega(x)\,dx,
$$
where $C_{(p_0,[\omega]_{A_1(\rn)})}$ is a positive constant independent
of $(F,G)$, but depends on $p_0$ and $[\omega]_{A_1(\rn)}$.
Assume that there exists a $q_0\in[p_0,\infty)$ such that $X^{1/q_0}$ is
a Banach function space and $\cm$ is bounded on $(X^{1/q_0})'$. Then there
exists a positive constant $C$ such that, for any $(F,G)\in\mathcal{F}$,
$$
\|F\|_{X}\leq C\|G\|_{X}.
$$
\end{lemma}

\subsection{Hardy type spaces \label{s2.4}}

Now we recall the notion of Hardy type spaces associated with ball quasi-Banach function spaces introduced in \cite{SHYY}.

\begin{definition}\label{HS}
Let $X$ be a ball quasi-Banach function space.
Let $\psi\in\cs(\rn)$ satisfy $\int_\rn\psi(x)\,dx\neq0$ and $b\in(0,\infty)$ sufficiently large.
Then the \emph{Hardy space} $H_X(\rn)$ associated with $X$ is defined by setting
$$
H_X(\rn):=\lf\{f\in\cs'(\rn):\ \|f\|_{H_X(\rn)}:=\lf\|M_b^{**}(f,\psi)\r\|_X<\infty\r\},
$$
where $M_b^{**}(f,\psi)$ is defined by setting, for any $x\in\rn$,
\begin{equation}\label{EqMb}
M_b^{**}(f,\psi)(x):=\sup_{(y,t)\in\rr_+^{n+1}}\frac{|\psi_t\ast f(x-y)|}{(1+|y|/t)^b}.
\end{equation}
\end{definition}

\begin{definition}\label{Demol}
Let $X$ be a ball quasi-Banach function space, $\epsilon\in(0,\infty)$, $q\in[1,\infty]$ and $d\in\zz_+$.
A measurable function $m$ is called an $(X,\,q,\,d,\,\epsilon)$-\emph{molecule} associated with some ball $B\subset\rn$ if
\begin{enumerate}
\item[(i)] for any $j\in\nn$, $\|m\|_{L^q(S_j(B))}\le2^{-j\epsilon}|S_j(B)|^\frac1q\|\mathbf{1}_B\|_{X}^{-1}$,
where $S_0:=B$ and, for any $j\in\nn$, $S_j(B):=(2^jB)\setminus(2^{j-1}B)$;
\item[(ii)] $\int_\rn m(x)x^\beta\,dx=0$ for any $\beta:=(\beta_1,\ldots,\beta_n)\in\zz_+^n$
with $|\beta|:=\beta_1+\cdots+\beta_n\le d$,
here and thereafter, for any $x:=(x_1,\ldots,x_n)\in\rn$,
$x^\beta:=x_1^{\beta_1}\cdots x_n^{\beta_n}$.
\end{enumerate}
\end{definition}

We also need the following molecular characterization of $H_{X}(\rn)$, which is just \cite[Theorem 3.9]{SHYY}.

\begin{lemma}\label{Mole}
Assume that $X$ is a ball quasi-Banach function space satisfying Assumption \ref{as1}
with $0<\theta< s\le1$ and Assumption \ref{as2} with some $q\in(1,\infty]$ and the same $s$ as in Assumption \ref{as1}.
Let $d\in\zz_+$ with $d\geq\lfloor n(1/\theta-1)\rfloor$ and $\epsilon\in(0,\infty)$ satisfy $\epsilon>n(1/\theta-1/q)$.
Then $f\in H_X(\rn)$ if and only if there exist a sequence $\{m_j\}_{j=1}^\infty$ of $(X,\,q,\,d,\,\epsilon)$-
molecules associated, respectively, with the balls $\{B_j\}_{j=1}^\infty\subset\BB$, and $\{\lambda_j\}_{j=1}^\infty\subset[0,\infty)$
satisfying
$$
\lf\|\lf\{\sum_{j=1}^\infty\lf(\frac{\lambda_j}{\|\mathbf1_{B_j}\|_X}\r)^s\mathbf1_{B_j}\r\}^{1/s}\r\|_X<\infty
$$
such that $\sum_{j=1}^\infty\lambda_jm_j$ converges in $\cs'(\rn)$. Moreover,
$$
\|f\|_{H_X(\rn)}\sim\lf\|\lf\{\sum_{j=1}^\infty\lf(\frac{\lambda_j}{\|\mathbf1_{B_j}\|_X}
\r)^s\mathbf1_{B_j}\r\}^{1/s}\r\|_X,
$$
where the positive equivalence constants are independent of $f$.
\end{lemma}

\section{Change of angles in  $X$-tent spaces\label{s3}}

In this section, we establish an estimate on
the change of angles in $X$-tent spaces.
Now we recall the notion of tent spaces associated with $X$.

\begin{definition}\label{cone}
For any $\alpha\in(0,\infty)$ and $x\in\rn$, let
$\Gamma_\alpha(x):=\{(y,t)\in\mathbb{R}_+^{n+1}:\ |x-y|<\alpha t\}$, which is
called the \emph{cone} of aperture $\alpha$
with vertex $x\in\rn$.
\end{definition}

Let $\alpha\in(0,\infty)$. For any measurable function
$F:\ \rr_+^{n+1}:=\rn\times(0,\infty)\to\cc$ and
$x\in\rn$, define
\begin{equation}\label{aa}
\ca^{(\alpha)}(F)(x):=\lf[\int_{\Gamma_\alpha(x)}|F(y,t)|^2\,\frac{dy\,dt}{t^{n+1}}\r]^\frac12,
\end{equation}
where $\Gamma_\alpha(x)$ is as in Definition \ref{cone}.
A measurable function $F$ is said to belong to the \emph{tent space}
$T_2^{p,\alpha}(\rr_+^{n+1})$, with $p\in(0,\infty)$, if $\|F\|_{T_2^{p,\alpha}(\rr_+^{n+1})}:=\|\ca^{(\alpha)}(F)\|_{L^p(\rn)}<\infty$.
Recall that Coifman et al. \cite{cms} introduced the tent space $T_2^{p,\alpha}(\mathbb{R}^{n+1}_+)$
for any $p\in (0,\fz)$ and $\alpha:=1$.
For any given ball quasi-Banach function space $X$,
the $X$-\emph{tent space} $T_X^\alpha(\rr_+^{n+1})$, with aperture $\alpha$,
is defined to be the set of all measurable
functions $F$ such that $\ca^{(\alpha)}(F)\in X$
and naturally equipped with the quasi-norm
$\|F\|_{T_X^{\alpha}(\rr_+^{n+1})}:=\|\ca^{(\alpha)}(F)\|_{X}$.

To prove Theorem \ref{pp}, we need the following inequality on the change of angles
in weighted Lebesgue spaces, which is a part of
\cite[Proposition 3.2]{MPA}.

\begin{lemma}\label{coa}
Let $\alpha,\ \beta\in(0,\infty)$ with $\alpha\le\beta$ and $q\in[1,\infty)$.
If $\omega\in A_q(\rn)$ and $p\in(0,2q]$,
then, for any measurable function
$F$ on $\mathbb{R}_+^{n+1}$,
$$
\int_\rn\lf|\ca^{(\beta)}(F)(x)\r|^p\omega(x)\,dx\leq
C\lf(\frac{\beta}{\alpha}\r)^{nq}\int_\rn\lf|\ca^{(\alpha)}(F)(x)\r|^p\omega(x)\,dx,
$$
where the positive constant $C$ is independent of $\alpha,\,\beta$ and $F$.
\end{lemma}

Using Lemma \ref{coa} and Theorem \ref{thet}, we have the following estimate on
the change of angles in $X$-tent spaces, which
plays a key role in the proof of Theorem \ref{Thgx} below.

\begin{theorem}\label{pp}
Let $X$ be a ball quasi-Banach function space.
Assume that there exists an $s\in(0,\infty)$ such that $X^{1/s}$ is a ball Banach function space
and $\cm$ is bounded on $(X^{1/s})'$. Then there exists a positive constant $C$ such that,
for any $\alpha\in[1,\infty)$ and any measurable function $F$ on $\mathbb{R}_+^{n+1}$,
\begin{equation}\label{Eqcc}
\lf\|\ca^{(\alpha)}(F)\r\|_X\le C\alpha^{\max\{\frac{n}{2},\frac{n}{s}\}}\lf\|\ca^{(1)}(F)\r\|_X.
\end{equation}
\end{theorem}

\begin{remark}
Assume that $X$ is a ball quasi-Banach function space satisfying Assumption \ref{as1}
with $0<\theta< s\le1$ and Assumption \ref{as2} with some $q\in(1,\infty]$ and the same $s$ as in Assumption \ref{as1}.
In this case, Wang et al. \cite[Lemma 2.20]{WYY} proved that
there exists a positive constant $C$ such that,
for any $\alpha\in[1,\infty)$ and any measurable function $F$ on $\mathbb{R}_+^{n+1}$,
$$
\lf\|\ca^{(\alpha)}(F)\r\|_X\le C\alpha^{\max\{\frac{n}{2}-\frac{n}{q}
+\frac{n}{\theta},\frac{n}{\theta}\}}\lf\|\ca^{(1)}(F)\r\|_X.
$$
Compared with the assumptions and the conclusions of \cite[Lemma 2.20]{WYY},
the assumptions in Theorem \ref{pp} are much weaker and the conclusions in Theorem \ref{pp} are more refined.
We should also point out that, in the case of $X:=L^p(\rn)$ with $p\in(0,\infty)$,
Theorem \ref{pp} coincides with the classical conclusions in \cite[Theorem 1.1]{A}.
\end{remark}

\begin{proof}[Proof of Theorem \ref{pp}]
Let $X$ be a ball quasi-Banach function space and $s\in(0,\infty)$.
Assume that $X^{1/s}$ is a ball Banach function space
and $\cm$ is bounded on $(X^{1/s})'$.
To show \eqref{Eqcc}, we consider the following two cases on $s$.

If $s\in(0,2]$, let
$$
\cf:=\left\{\left(\ca^{(\alpha)}(F),
\alpha^{\frac{n}{s}}\ca^{(1)}(F)\right):\ \alpha\in[1,\infty),\ F\in\mathscr M(\rr_+^{n+1})\right\}.
$$
Then, by Lemma \ref{coa}, we know that, for any given $\omega\in A_1(\rn)$
and for any
$(\ca^{(\alpha)}(F),
\alpha^{\frac{n}{s}}\ca^{(1)}(F))\in\mathcal{F}$,
$$
\int_{\rn}\lf|\ca^{(\alpha)}(F)(x)\r|^s\omega(x)\,dx\lesssim
\alpha^{n}\int_{\rn}\lf|\ca^{(1)}(F)(x)\r|^s\omega(x)\,dx
\sim\int_{\rn}\lf|\alpha^{\frac{n}{s}}\ca^{(1)}(F)(x)\r|^s\omega(x)\,dx,
$$
which, together with the assumptions that $X^{1/s}$ is a ball Banach function space and
$\cm$ is bounded on $(X^{1/s})'$, and Lemma \ref{thet}, further implies that,
for any $\alpha\in[1,\infty)$ and $F\in\mathscr M(\rr_+^{n+1})$,
\begin{equation}\label{eq3.5}
\lf\|\ca^{(\alpha)}(F)\r\|_{X}\lesssim\alpha^{\frac{n}{s}}\lf\|\ca^{(1)}(F)\r\|_{X}.
\end{equation}
This is the desired estimate.

If $s\in(2,\infty)$, let
$$
\cf:=\left\{\left(\ca^{(\alpha)}(F),\alpha^{\frac{n}{2}}\ca^{(1)}(F)\right):\
\alpha\in(1,\infty),\ F\in\mathscr M(\rr_+^{n+1})\right\}.
$$
Since $s/2>1$, by the definition of $A_p(\rn)$ weights,
we know that, for any $\omega\in A_1(\rn)$, $\omega\in A_{s/2}(\rn)$ holds true.
From this and Lemma \ref{coa}, we easily deduce that, for any given $\omega\in A_1(\rn)$
and for any
$(\ca^{(\alpha)}(F),
\alpha^{\frac{n}{2}}\ca^{(1)}(F))\in\mathcal{F}$,
$$
\int_\rn\left|\ca^{(\alpha)}(F)(x)\right|^s\omega(x)\,dx\lesssim \alpha^{\frac{ns}{2}}\int_\rn\lf|\ca^{(1)}(F)(x)\r|^s\omega(x)\,dx
\sim\int_\rn\lf|\alpha^{\frac{n}{2}}\ca^{(1)}(F)(x)\r|^s\omega(x)\,dx,
$$
which, combined with the assumptions that $X^{1/s}$ is a ball Banach function space and
$\cm$ is bounded on $(X^{1/s})'$, and Lemma \ref{thet}, further implies that,
for any $\alpha\in[1,\infty)$ and $F\in\mathscr M(\rr_+^{n+1})$,
\begin{equation}\label{eq3.6}
\lf\|\ca^{(\alpha)}(F)\r\|_{X}\lesssim \alpha^{\frac{n}{2}}\lf\|\ca^{(1)}(F)\r\|_{X}.
\end{equation}
By \eqref{eq3.5} and \eqref{eq3.6}, we conclude that, for any $\alpha\in[1,\infty)$ and $F\in\mathscr M(\rr_+^{n+1})$,
$$
\lf\|\ca^{(\alpha)}(F)\r\|_X\lesssim\alpha^{\max\{\frac{n}{2},\frac{n}{s}\}}\lf\|\ca^{(1)}(F)\r\|_X,
$$
which is also the desired estimate and hence then
completes the proof of Theorem \ref{pp}.
\end{proof}

\section{Littlewood--Paley function characterizations\label{s4}}

In this section,
we establish various Littlewood--Paley function characterizations of $H_X(\rn)$,
including its characterizations via the Lusin area function, the Littlewood--Paley
$g$-function and the Littlewood--Paley $g_\lambda^\ast$-function, respectively,
in \S\ref{s4.1}, \S\ref{s4.2} and \S\ref{s4.3} below.

In what follows, the symbol $\vec 0_n$ denotes the \emph{origin} of $\rn$ and,
for any $\phi\in\cs(\rn)$, $\widehat\phi$ denotes its \emph{Fourier transform}
which is defined by setting, for any $\xi\in\rn$,
$$
\widehat\phi(\xi):=\int_\rn e^{-2\pi ix\xi}\phi(x)\,dx.
$$
For any $f\in\cs'(\rn)$, $\widehat f$ is defined by setting, for any $\varphi\in\mathcal{S}(\rn)$,
$\la\widehat f,\varphi\ra:=\la f,\widehat\varphi\ra$; also, for any $f\in\mathcal{S}(\rn)$
[resp., $\mathcal{S}'(\rn)$],
$f^{\vee}$ denotes its \emph{inverse Fourier transform} which is defined by setting,
for any $\xi\in\rn$, $f^{\vee}(\xi):=\widehat{f}(-\xi)$ [resp., for any $\varphi\in\mathcal{S}(\rn)$,
$\la f^{\vee},\varphi\ra:=\la f,\varphi^{\vee}\ra$].

\begin{definition}\label{de4.1}
Let $\varphi\in\cs(\rn)$ satisfy $\widehat{\varphi}(\vec0_n)=0$ and
assume that, for any $\xi\in\rn\setminus\{\vec0_n\}$,
there exists a $t\in(0,\infty)$ such that $\widehat\varphi(t\xi)\neq0$.
For any distribution $f\in\cs'(\rn)$, the \emph{Lusin-area function} $S(f)$
and the \emph{Littlewood-Paley} $g_\lambda^\ast$-\emph{function} $g_\lambda^\ast(f)$ with
any given $\lambda\in(0,\infty)$ are defined, respectively, by setting, for any $x\in\rn$,
\begin{equation}\label{deaf}
S(f)(x):=\lf\{\int_{\Gamma(x)}\lf|\varphi_t\ast f(y)\r|^2\,\frac{dy\,dt}{t^{n+1}}\r\}^\frac12
\end{equation}
and
\begin{equation}\label{d22}
g_\lambda^\ast(f)(x):=\lf\{\int_0^\infty\int_\rn\lf(\frac{t}{t+|x-y|}\r)^{\lambda n}\lf|\varphi_t\ast f(x)\r|^2\,\frac{dy\,dt}{t^{n+1}}\r\}^\frac12,
\end{equation}
where, for any $x\in\rn$, $\Gamma(x)$ is as in Definition \ref{cone} and, for any $t\in(0,\infty)$ and $x\in\rn$,
$\varphi_t(x):=t^{-n}\varphi(x/t)$.
\end{definition}

\begin{definition}\label{de4.2}
Let $\varphi\in\mathcal{S}(\rn)$ satisfy $\widehat{\varphi}(\vec 0_n)=0$
and assume that, for any $x\in\rn\setminus\{\vec 0_n\}$,
there exists a $j\in \zz$ such that $\widehat\varphi(2^{j}x)\not=0$.
For any $f\in\cs'(\rn)$, the \emph{Littlewood--Paley $g$-function $g(f)$}
is defined by setting, for any $x\in\rn$,
\begin{equation}\label{degf}
g(f)(x):=\lf[\int_0^\infty|f\ast\varphi_t(x)|^2\,\frac{dt}{t}\r]^{1/2}.
\end{equation}
\end{definition}
\begin{remark}\label{wa}
\begin{itemize}
\item[(i)]
The way to define the Littlewood--Paley functions in Definitions \ref{de4.1} and \ref{de4.2}
is the same as in \cite[(3.2), (3.3) and (3.4)]{WYYZ}.
Observe that,
in Definitions \ref{de4.1} and \ref{de4.2}, we did not assume that $\varphi$ is radial and
has compact support and hence, compared with the assumptions required in \cite{HYY,HLYY,LHY,NS,ZSY,ZYYW},
the assumptions here in both cases are quite weaker.

\item[(ii)] In all these Littlewood--Paley function characterizations of $H_X(\rn)$,
we only need $\widehat{\varphi}(\vec 0_n)=0$, namely, $\varphi$ has a zero order vanishing moment.
Compared with
all the known results on the Littlewood--Paley function characterizations on function
spaces (see, for instance, \cite{HYY,HLYY,LHY,NS,ZSY,ZYYW}), this assumption on the vanishing moment of $\varphi$ is also minimal.
\end{itemize}
\end{remark}

\subsection{Characterization by the Lusin area function\label{s4.1}}

In this subsection, borrowing some ideas from the proof of \cite[Theorem 3.12]{WYYZ},
we characterize the Hardy type space $H_X(\rn)$ by the Lusin-area function.

Let $\alpha\in(0,\infty)$. For any ball $B(x,r)\subset\rn$ with $x\in\rn$ and $r\in(0,\infty)$, let
$$
T_\alpha(B):=\lf\{(y,t)\in\rr^{n+1}_+:\ 0<t<\frac{r}{\alpha},\ |y-x|<r-\alpha t\r\}.
$$
When $\alpha=1$, we denote $T_\alpha(B)$ simply by $T(B)$.

\begin{definition}\label{1q1}
Let $X$ be a ball quasi-Banach function space, $p\in(1,\infty)$ and $\alpha\in(0,\infty)$.
A measurable function $a:\ \rr_+^{n+1}\to\cc$ is called a
\emph{$(T_X,p)$-atom}, supported in $T(B)$, if there
exists a ball $B\subset\rn$ such that
\begin{itemize}
\item[(i)] $\supp(a):=\{(x,t)\in\rr_+^{n+1}:\ a(x,t)\neq0\}\subset T(B)$,
\item[(ii)] $\|a\|_{T_2^{p,1}(\rr_+^{n+1})}\le|B|^{1/p}/\|\mathbf{1}_B\|_{X}$.
\end{itemize}
Moreover, if $a$ is a $(T_X,p)$-atom for any $p\in(1,\infty)$,
then $a$ is called a \emph{$(T_X,\infty)$-atom}.
\end{definition}

To establish the Lusin area function characterization of $H_X(\rn)$, we need the
following lemma which is just \cite[Proposition 4.9]{SHYY}.

\begin{lemma}\label{lt}
Let $F:\ \rr^{n+1}_+\to\cc$ be a measurable function.
Assume that $X$ is a ball quasi-Banach function space satisfying Assumptions \ref{as1}
and \ref{as2} with the same $s\in(0,1]$.
Then $f\in T_X^1(\rr^{n+1}_+)$ if and only if there exist a sequence $\{\lambda_j\}_{j=1}^\infty\subset [0,\infty)$
and a sequence $\{A_j\}_{j=1}^\infty$ of $(T_X^1,\infty)$-atoms supported, respectively, in
$\{T(B_j)\}_{j=1}^\infty$ such that, for almost every $(x,t)\in\rr_+^{n+1}$,
\begin{equation*}
F(x,t)=\sum_{j=1}^\infty\lambda_jA_j(x,t)\quad\text{and}\quad|F(x,t)|=\sum_{j=1}^\infty\lambda_j|A_j(x,t)|
\end{equation*}
pointwisely, and
$$
\lf\|\lf\{\sum_{j=1}^\infty\lf(\frac{\lambda_j}{\|\mathbf1_{B_j}\|_X}\r)^s\mathbf1_{B_j}\r\}^{1/s}\r\|_X<\infty.
$$
Moreover,
$$
\|f\|_{T_X^1(\rr_+^{n+1})}\sim\Lambda\lf(\{\lambda_jA_j\}_{j\in\nn}\r)
$$
where the positive equivalence constants are independent of $f$, but depend on $s$.
\end{lemma}

Combining Calder\'on \cite[Lemma 4.1]{C1975} and Folland and Stein \cite[Theorem 1.64]{FoS}
(see also \cite[Lemma 4.6]{YYYZ}),
the following Calder\'on reproducing formula was obtained in \cite[Lemma 4.4]{ZWYY}.
Recall that $f\in\cs'(\rn)$ is said to \emph{vanish weakly at infinity} if, for any $\phi\in\cs(\rn)$,
$f\ast\phi_t\to0$ in $\cs'(\rn)$ as $t\to\infty$ (see, for instance, \cite[p.\,50]{FoS}),
here and thereafter, for any $t\in(0,\infty)$ and function $\phi$ on $\rn$, we always let
$\phi_t(\cdot):=t^{-n}\phi(\cdot/t)$.
In what follows, the \emph{symbol $\epsilon\to0^+$} means that
$\epsilon\in(0,\infty)$ and $\epsilon\to0$, and
the \emph{symbol $C_c^\fz(\rn)$} denotes the set of all infinitely differentiable
functions with compact supports.

\begin{lemma}\label{Le47}
Let $\phi$ be a Schwartz function and assume that, for any $x\in\rn\setminus\{\vec 0_n\}$,
there exists a $t\in (0,\fz)$ such that $\widehat\phi(tx)\not=0$. Then there
exists a $\psi\in\cs(\rn)$ such that $\wh\psi\in C^\fz_c(\rn)$ with its support
away from $\vec 0_n$, $\wh\phi\wh\psi\ge 0$ and, for any $x\in\rn\setminus\{\vec 0_n\}$,
$$\int^\fz_0\wh\phi(tx)\wh\psi(tx)\,\frac {dt}t=1.$$
Moreover, for any $f\in\cs'(\rn)$, if $f$ vanishes weakly at infinity,
then
$$
f=\int_0^\infty f\ast\phi_t\ast\psi_t\,\frac{dt}{t}\quad\text{in}\quad\cs'(\rn),
$$
namely,
$$
f=\lim_{\substack{\epsilon\to0^+\\ A\to\infty}}
\int_\epsilon^A f\ast\phi_t\ast\psi_t\,\frac{dt}{t}\quad\text{in}\quad\cs'(\rn).
$$
\end{lemma}

To obtain the Lusin area function characterization of $H_X(\rn)$, we also need the
following technical lemma.

\begin{lemma}\label{embed}
Let $X$ be a ball quasi-Banach function space.
Assume that there exists an $s\in(0,\infty)$ such that $X^{1/s}$ is a ball Banach function space
and $\cm$ is bounded on $(X^{1/s})'$.
Then there exists an $\epsilon\in(0,1)$ such that
$X$
continuously embeds into $L_\omega^s(\rn)$ with $\omega:=[\cm(\mathbf1_{B(\vec0_n,1)})]^\epsilon$.
\end{lemma}
\begin{proof}
By \cite[Lemma 2.15(ii)]{SHYY} and the fact that $\cm$ is bounded on $(X^{1/s})'$,
we know that there exists an $\eta\in(1,\infty)$ such that $\cm^{(\eta)}$ is bounded on $(X^{1/s})'$.
Let $\epsilon\in(1/\eta,1)$ and $\omega:=[\cm(\mathbf1_{B(\vec0_n,1)})]^\epsilon$.
To show this lemma,
it suffices to prove that, for any $f\in X$,
\begin{equation}\label{lw_m}
\|f\|_{L^s_\omega(\rn)}\lesssim \|f\|_{X}.
\end{equation}
Indeed, from \cite[(2.1.6)]{G1}, we deduce that, for any $x\in\rn$, $\cm(\mathbf1_{B(\vec0_n,1)})(x)\sim(|x|+1)^{-n}$,
which implies that,
for any $x\in B({\vec 0_n},2)$,
\begin{equation}\label{est_c}
\omega(x)=\lf[\cm\lf(\mathbf1_{B(\vec0_n,1)}\r)(x)\r]^\epsilon\lesssim 1
\end{equation}
and, for any $k\in\nn$ and $x\in B({\vec 0_n},2^{k+1})\setminus B({\vec 0_n},2^k)$,
\begin{equation}\label{est_r}
\omega(x)=\lf[\cm\lf(\mathbf1_{B(\vec0_n,1)}\r)(x)\r]^\epsilon\lesssim2^{-\epsilon kn}.
\end{equation}
Combining \eqref{est_c}, \eqref{est_r}, the H\"older inequality and the fact
that $X^{1/s}$ is a ball Banach function space,
we conclude that, for any $f\in X$,
\begin{align*}
\int_\rn\lf|f(x)\r|^s\omega(x)\,dx&=\int_{B(\vec0_n,2)}\lf|f(x)\r|^s\omega(x)\,dx
+\sum_{k=1}^\infty\int_{B(\vec0_n,2^{k+1})\setminus B(\vec0_n,2^k)}\cdots\\
&\lesssim\lf\||f|^s\r\|_{X^{1/s}}\lf\|\mathbf1_{B(\vec0_n,2)}\r\|_{(X^{1/s})'}
+\sum_{k=1}^\infty\int_{B(\vec0_n,2^{k+1})\setminus B(\vec0_n,2^k)}|f(x)|^s2^{-\epsilon kn}\,dx\\
&\lesssim\lf\||f|^s\r\|_{X^{1/s}}\lf\|\mathbf1_{B(\vec0_n,2)}\r\|_{(X^{1/s})'}
+\sum_{k=1}^\infty2^{-\epsilon kn}\lf\||f|^s\r\|_{X^{1/s}}\lf\|\mathbf1_{B(\vec0_n,2^{k+1})}\r\|_{(X^{1/s})'},
\end{align*}
which, together with
$\mathbf1_{B(\vec0_n,2^k)}\lesssim2^{kn/\eta}\cm^{(\eta)}(1_{B(\vec0_n,1)})$ for any $k\in\nn$
and the fact that $\cm^{(\eta)}$ is bounded on $(X^{1/s})'$, further implies that
\begin{align*}
\int_\rn\lf|f(x)\r|^s\omega(x)\,dx&\lesssim
\sum_{k=1}^\infty2^{-\epsilon kn}\lf\|f\r\|_{X}^s\lf\|2^{kn/\eta}\cm^{(\eta)}(1_{B(\vec0_n,1)})\r\|_{(X^{1/s})'}\\
&\lesssim\sum_{k=1}^\infty2^{-(\epsilon-1/\eta)
kn}\lf\|f\r\|_{X}^s\lf\|\mathbf1_{B(\vec0_n,1)}\r\|_{(X^{1/s})'}
\lesssim\|f\|_X^s,
\end{align*}
which implies that \eqref{lw_m} holds true and hence completes the proof of Lemma \ref{embed}.
\end{proof}

Now we recall the notion of atoms associated with $X$, which is just \cite[Definition 3.5]{SHYY}.

\begin{definition}\label{Deatom}
Let $X$ be a ball quasi-Banach function space, $q\in(1,\infty]$ and $d\in\zz_+$.
Then a measurable function $a$ on $\rn$ is called an $(X,\,q,\,d)$-\emph{atom}
if there exists a ball $B\in\BB$ such that
\begin{itemize}
\item[(i)] $\supp a:=\{x\in\rn:\ a(x)\neq0\}\subset B$;
\item[(ii)] $\|a\|_{L^q(\rn)}\le\frac{|B|^{1/q}}{\|\mathbf{1}_B\|_X}$;
\item[(iii)] $\int_{\rn}a(x)x^\alpha\,dx=0$ for any
$\alpha\in\zz_+^n$ with $|\alpha|\le d$.
\end{itemize}
\end{definition}

In what follows, for any $t\in(0,\fz)$, the \emph{symbol $e^{t\Delta}f$} denotes
the heat extension of $f\in\cs'(\rn)$, namely, for any $x\in\rn$,
$$e^{t\Delta}f(x):=\left<f,\frac{1}{(4\pi t)^{n/2}}
\exp\left(-\frac{|x-\cdot|^2}{4t}\right)\right>.
$$

\begin{theorem}\label{Tharea}
Assume that $X$ is a ball quasi-Banach function space satisfying Assumptions \ref{as1}
and \ref{as2} with the same $s\in(0,1]$.
Then $f\in H_X(\rn)$ if and only if $f\in\cs'(\rn)$, $f$ vanishes weakly at infinity and $\|S(f)\|_X<\infty$.
Moreover, for any $f\in H_X(\rn)$,
$$
\|f\|_{H_X(\rn)}\sim\|S(f)\|_X,
$$
where the positive equivalence constants are independent of $f$.
\end{theorem}

\begin{remark}\label{gap}
If $\varphi$ appearing in the definition of $S(f)$ as in \eqref{deaf} satisfies that
$\mathbf1_{B(\vec 0_n,4)\setminus B(\vec0_n,2)}\le\widehat\varphi\le\mathbf1_{B(\vec 0_n,8)\setminus B(\vec0_n,1)}$,
then, in this case, Theorem \ref{Tharea} coincides with \cite[Theorem 3.21]{SHYY}. 
We point out that there exists a gap in lines 1-17 of \cite[p.\,52]{SHYY} which appears in the proof 
of \cite[Theorem 3.21]{SHYY} when Sawano et al. proved that $f$ vanishes weakly at infinity.
We seal this gap in the below proof of Theorem \ref{Tharea} by using Lemma \ref{embed}.
\end{remark}

\begin{proof}[Proof of Theorem \ref{Tharea}]
Let $\theta$ and $s$ be as in
Assumption \ref{as1}. 

We first prove the necessity. To this end, let $f\in H_X(\rn)$
and we need to show that $f$ vanishes weakly at infinity,
which seals the gap mentioned in Remark \ref{gap}.
From Lemma \ref{embed}, we deduce that there exists an $\epsilon\in(0,1)$ such that
$X$ continuously embeds into $L_\omega^s(\rn)$ with $\omega:=[\cm(\mathbf1_{B(\vec0_n,1)})]^\epsilon$,
which implies that $f\in H_\omega^s(\rn)$,
where $H_\omega^s(\rn)$ is the weighted Hardy space as in Definition \ref{HS} with $X$ replaced by $L_\omega^s(\rn)$.
By \cite[Theorem 7.2.7]{G1},
we know that $\omega\in A_1(\rn)$, which, combined with \cite[Remark 2.4(b) and Remark 2.6(b)]{WYY},
implies that $L_\omega^s(\rn)$
satisfies all the assumptions of \cite[Theorems 3.7]{SHYY}.
Let
$d\geq\lfloor n(1/\theta-1)\rfloor$.
Then, using \cite[Theorem 3.7]{SHYY} and the fact that $f\in H_\omega^s(\rn)$, we conclude that
there exist a sequence $\{a_j\}_{j\in\nn}$ of $(L_\omega^s(\rn),\infty,d)$-atoms 
supported, respectively, in balls
$\{B_j\}_{j\in\nn}$ and
a sequence $\{\lambda_j\}_{j\in\nn}\subset[0,\fz)$ such that
\begin{equation}\label{Eq73}
f=\sum_{j\in\nn}\lambda_ja_j\quad\text{in}\quad H_\omega^s(\rn)
\end{equation}
and
\begin{equation}\label{Eq75}
\lf\|\lf\{\sum_{j\in\nn}\lf[\frac{\lambda_j\mathbf{1}_{B_j}}{\|\mathbf{1}_{B_j}
\|_{L_\omega^s(\rn)}}\r]^{s}\r\}^\frac1{s}\r\|_{L_\omega^s(\rn)}
\lesssim\|f\|_{H_\omega^s(\rn)}<\infty,
\end{equation}
where an $(L_\omega^s(\rn),\infty,d)$-atom is as in Definition \ref{Deatom}
with $X$ replaced by $L_\omega^s(\rn)$.
Take $\varphi,\,\psi\in{\mathcal S}(\rn)$.
Then, by the fact that $f=\sum_{j\in\nn}\lambda_ja_j$ in $H_\omega^s(\rn)$,
we find that, for any $t\in(0,\fz)$ and finite set ${\mathfrak F}\subset\nn$,
\begin{align}\label{eq:4-48}
\int_{{\mathbb R}^n}\psi_t\ast f(x)\varphi(x)\,dx&=\sum_{j\in{\mathfrak F}}\lambda_j
\int_{{\mathbb R}^n}\psi_t\ast a_j(x)\varphi(x)\,dx+\sum_{j \in{\nn}\setminus{\mathfrak F}}\cdots.
\end{align}
Notice that, for any $j\in\nn$, $a_j\in L^\infty(\rn)$, which implies that $a_j$ vanishes weakly at infinity.
Thus, for any given $\varepsilon\in(0,\fz)$ and finite set ${\mathfrak F}\subset\nn$,
there exists $t_{\varepsilon,\,{\mathfrak F}}\in(0,\fz)$ such that, for any
$t\in(t_{\varepsilon,\,{\mathfrak F}},\infty)$,
\begin{equation}\label{eq:4-49}
\left|\sum_{j \in {\mathfrak F}}\lambda_j\int_{{\mathbb R}^n}\psi_t\ast a_j(x)\varphi(x)\,dx
\right|<\varepsilon.
\end{equation}
Moreover, for any $j\in\nn$, using the fact that $a_j$ is an $(L_\omega^s(\rn),\infty,d)$-atom,
similarly to the proof of \cite[(4.16)]{SHYY}, we conclude that, for any $t,\,t_1\in(0,\fz)$ and $x\in\rn$,
$$\lf|e^{-t_1\Delta}\psi_t\ast a_j(x)\r|\lesssim
\frac{1}{\|\mathbf{1}_{B_j}\|_{L_\omega^s(\rn)}}\cm^{(\theta)}(\mathbf{1}_{B_j})(x),$$
which, together with \cite[Corollary 3.2]{SHYY},
$s\in(0,1]$ and Assumption \ref{as1}, implies that
\begin{align*}
&\left|\sum_{j\in{\nn} \setminus{\mathfrak F}}
\lambda_j\int_{{\mathbb R}^n}\psi_t\ast a_j(x)\varphi(x)\,dx\right|\\
&\quad=
\lf|\lf(\sum_{j\in{\nn} \setminus{\mathfrak F}}
\lambda_j\psi_t\ast a_j\r)\ast[\varphi(-\cdot)](\vec 0_n)\r|
\lesssim\left\|\sum_{j\in{\nn}\setminus{\mathfrak F}}\lambda_j\psi_t\ast a_j\right\|_{
H_\omega^s(\rn)}\\
&\quad\sim\left\|\sum_{j\in{\nn}\setminus {\mathfrak F}}
\lz_j\sup_{t_1\in(0,\fz)}\left|e^{-t_1\Delta}\psi_t\ast a_j\right|\right\|_{L_\omega^s(\rn)}
\lesssim\left\|\sum_{j\in{\nn}\setminus {\mathfrak F}}\frac{\lambda_j}{\|\mathbf{1}_{B_j}\|_{L_\omega^s(\rn)}}\cm^{(\theta)}(\mathbf{1}_{B_j})
\right\|_{L_\omega^s(\rn)}\\
&\quad\ls\left\|\lf\{\sum_{j\in{\nn}\setminus {\mathfrak F}}
\lf[\frac{\lambda_j}{\|\mathbf{1}_{B_j}\|_{L_\omega^s(\rn)}}\cm^{(\theta)}(\mathbf{1}_{B_j})\r]^s\r\}^{1/s}
\right\|_{L_\omega^s(\rn)}
\lesssim\left\|\lf\{\sum_{j\in{\nn}\setminus{\mathfrak F}}
\lf(\frac{\lambda_j}{\|\mathbf{1}_{B_j}\|_{L_\omega^s(\rn)}}\r)^{s}\mathbf{1}_{B_j}\r\}^{1/s}
\right\|_{L_\omega^s(\rn)}.
\end{align*}
From this, the fact that
$$\left\|\left\{\sum_{j\in{\nn}}
\left(\frac{\lambda_j}{\|\mathbf{1}_{B_j}\|_{L_\omega^s(\rn)}}
\right)^s\mathbf{1}_{B_j}\right\}^{\frac{1}{s}}\right\|_{{L_\omega^s(\rn)}}<\infty,
$$
and the dominated convergence theorem,
we deduce that, for $\varepsilon\in(0,\fz)$ as in \eqref{eq:4-49}, there exists a finite
set $\mathfrak{F}\subset\nn$ such that, for any $t\in(0,\fz)$,
$$\left|\sum_{j\in{\nn} \setminus{\mathfrak F}}
\lambda_j\int_{{\mathbb R}^n}\psi_t\ast a_j(x)\varphi(x)\,dx\right|<\varepsilon,$$
which, combined with \eqref{eq:4-48} and \eqref{eq:4-49}, further implies that $f$ vanishes weakly at infinity.

Then, by an argument similar to that used in the proof of \cite[Theorem 3.21]{SHYY},
we obtain the necessity of this theorem and we omit the details.

Now we prove the sufficiency.
To this end, let $f\in\cs'(\rn)$ vanish weakly at infinity and satisfy that $S(f)\in X$.
By this and Lemma \ref{Le47}, we conclude that
there exists a $\psi\in\cs(\rn)$ such that
\begin{equation}\label{67mm}
\supp(\widehat{\psi})\subset B(\vec{0}_n,b) \setminus B(\vec{0}_n,a),
\end{equation}
where $b,\ a\in(0,\infty)$ and $b>a$, and
\begin{equation}\label{66mm}
f=\int_0^\infty f\ast\varphi_t\ast\psi_t\,\frac{dt}{t}\quad\text{in}\quad\cs'(\rn)
\end{equation}
with $\varphi$ as in Definition \ref{de4.1}.
It remains to prove that $f\in H_X(\rn)$.
For any $(x,t)\in{\mathbb R}^{n+1}_+$, let $F(x,t):=f\ast\varphi_t(x)$. Then,
by the fact that $\|S(f)\|_X<\infty$, we know that
$F\in T_X^1(\rr^{n+1}_+)$. From this and Lemma \ref{lt}, it follows that
there exist a sequence $\{A_j\}_{j=1}^\infty$ of $(T_X^1,\infty)$-atoms and $\{\lambda_j\}_{j=1}^\infty\subset[0,\infty)$
such that, for almost every $(x,t)\in\rr^{n+1}_+$,
\begin{align}\label{456}
F(x,t)=\sum_{j\in\nn}\lambda_{j}A_{j}(x,t)\quad\text{and}\quad|F(x,t)|=\sum_{j\in\nn}\lambda_{j}|A_{j}(x,t)|
\end{align}
pointwisely on $\rr_+^{n+1}$ and, for some $s\in(0,1]$,
\begin{align}\label{457}
\lf\|\lf\{\sum_{j=1}^\infty\lf(\frac{\lambda_j}{\|\mathbf1_{B_j}\|_X}\r)^s\mathbf1_{B_j}\r\}^{1/s}\r\|_X
\lesssim\lf\|F\r\|_{T_X^1(\rr_+^{n+1})}\sim\lf\|S(f)\r\|_{X}.
\end{align}

For any $j\in\nn$ and $x\in\rn$, let
\begin{equation}\label{kkk}
A_j(x):=\int_0^{\infty}\int_\rn A_j(y,t)\psi_t(x-y)\,\frac{dy\,dt}{t}.
\end{equation}
Similarly to the proof of \cite[Lemma 4.8]{HYY}, we conclude
that, up to a harmless constant multiple, $\{A_j\}_{j=1}^\infty$ is a sequence of $(X, q, d, \epsilon)$-
molecules associated, respectively, with balls $\{B_j\}_{j=1}^\infty$, where $q$, $d$ and $\epsilon$ are as in Lemma \ref{Mole}.
Repeating the
argument used in \cite[(3.27)]{WYYZ}, we find that $\sum_{j=1}^\infty\lambda_jA_j$
converges in $\cs'(\rn)$.
Using this, \eqref{457} and Lemma \ref{Mole}, we then obtain $\sum_{j=1}^\infty\lambda_jA_j\in H_X(\rn)$ and
\begin{equation}\label{kk2}
\lf\|\sum_{j=1}^\infty\lambda_jA_j\r\|_{H_X}
\lesssim\lf\|\lf\{\sum_{j=1}^\infty\lf(\frac{\lambda_j}{\|\mathbf1_{B_j}\|_X}\r)^s\mathbf1_{B_j}\r\}^{1/s}\r\|_X
\lesssim\lf\|S(f)\r\|_{X}.
\end{equation}
Let $g:=\sum_{j=1}^\infty\lambda_{j}A_{j}$ in $\cs'(\rn)$.
Then $g\in H_X(\rn)$. From this and the necessity of this theorem, it follows that $g$ vanishes weakly at infinity.
Using this and repeating
the proof of \cite[(3.30)]{WYYZ} , we find that
\begin{equation}\label{zz}
f=g\quad\text{in}\quad\cs'(\rn)
\end{equation}
holds true.
By \eqref{zz} and \eqref{kk2},
we conclude that $f\in H_X(\rn)$ and $\|f\|_{WH_X(\rn)}
\lesssim\|S(f)\|_{WX}$,
which completes the proof of the sufficiency and hence
of Theorem \ref{Tharea}.
\end{proof}

\subsection{Characterization by the Littlewood--Paley $g_\lambda^\ast$-Function\label{s4.2}}

In this subsection, we establish the Littlewood--Paley $g_\lambda^\ast$-function characterization
of $H_X(\rn)$.

Let $X$ be a ball quasi-Banach function space and
\begin{equation}\label{eqsup}
r_+:=\sup\lf\{s\in(0,\infty):\
X\ \text{satisfies Assumption \ref{as2} for this}\ s\ \text{and some}\ q\in(s,\infty)\r\}.
\end{equation}

\begin{theorem}\label{Thgx}
Assume that $X$ is a ball quasi-Banach function space satisfying Assumptions \ref{as1}
and \ref{as2} with the same $s\in(0,1]$.
Let $r_+$ be as in \eqref{eqsup} and $\lambda\in(\max\{1,{2}/{r_+}\},\infty)$.
Then $f\in H_X(\rn)$ if and only if $f\in\cs'(\rn)$, $f$ vanishes weakly at infinity and
$\|g_\lambda^\ast(f)\|_X<\infty$.
Moreover, for any $f\in H_X(\rn)$,
$$
\|f\|_{H_X(\rn)}\sim\|g_\lambda^\ast(f)\|_X,
$$
where the positive equivalence constants are independent of $f$.
\end{theorem}
\begin{remark}
Assume that $X$ is a ball quasi-Banach function space satisfying Assumption \ref{as1}
with $0<\theta< s\le1$ and Assumption \ref{as2} with some $q\in(1,\infty]$ and the same $s$ as in Assumption \ref{as1}. Let $r_+$ be as in \eqref{eqsup}.
We point out that
the $g_\lambda^\ast$-function characterization in Theorem \ref{Thgx}
widens the range of $\lambda\in(\max\{2/\theta,2/\theta+1-2/q\},\infty)$
in \cite[Theorem 2.10(ii)]{WYY} into $\lz\in(\max\{1,2/r_+\},\infty)$.
In the case of $X:=L^p(\rn)$ with $p\in(0,\infty)$,
the range of $\lambda$ in Theorem \ref{Thgx} coincides with the best
known one, namely, $(\max\{1,2/p\},\infty)$ in \cite{FoS}.
\end{remark}

\begin{proof}[Proof of Theorem \ref{Thgx}]
By Theorem \ref{Tharea} and the fact that, for any $f\in\cs'(\rn)$,
$S(f)\le g_\lambda^\ast(f)$,
we easily obtain the sufficiency of this theorem and still
need to show its necessity.

To this end, let $f\in H_X(\rn)$.
By Theorem \ref{Tharea}, we know that $f$
vanishes weakly at infinity. Moreover, for any $x\in\rn$, we have
\begin{align*}
g_\lambda^*(f)(x)&\le\lf\{\int_0^{\infty}\int_{|x-y|<t}\lf(\frac{t}{t+|x-y|}\r)
^{\lambda n}|\varphi_t\ast f(y)|^{2}\,\frac{dy\,dt}{t^{n+1}}\r.\\
&\hs+\lf.\sum_{m=0}^\infty\int_0^{\infty}\int_{2^mt\le|x-y|<2^{m+1}t}\cdots\r\}^{\frac{1}{2}}\\
&\le\ca^{(1)}\lf(F\r)(x)+\sum_{m=0}^\infty2^{\frac{-\lambda nm}{2}}
\ca^{(2^{m+1})}\lf(F\r)(x),
\end{align*}
where $F(x,t):=\varphi_t\ast f(x)$ for any $x\in\rn$ and $t\in(0,\infty)$.

For any given $\lambda\in(\max\{1,2/r_+\},\infty)$, there exists an
$r\in(0,r_+]$ such that $X$ satisfies Assumption \ref{as2}
for this $r$ and some $q\in(r,\infty)$,
and $\lambda\in(\max\{1,2/r\},\infty)$.
Let $\nu\in(0,\min\{1,r\}]$.
Then, by the assumption that $X^{1/r}$ is a ball Banach function space, and Lemma \ref{Re2.6},
we know that $X^{1/\nu}$ is also a ball Banach function space. Since $X$ satisfies Assumption \ref{as2},
from Lemma \ref{Le2.9} and the fact that $(q/r)'\in(1,\infty)$, we deduce that
$\cm$ is bounded on $(X^{1/r})'$, which, combined with the fact that $X^{1/r}$ is a ball Banach function space,
and Theorem \ref{pp}, implies that, for any $m\in\nn$,
$$
\lf\|\ca^{(2^{m+1})}\lf(F\r)\r\|_X\lesssim\max\lf\{2^{\frac{mn}{2}},2^\frac{mn}{r}\r\}
\lf\|\ca^{(1)}\lf(F\r)\r\|_{X}.
$$
By this, the fact that $X^{1/\nu}$ is a ball Banach function space,
$\lambda\in(\max\{\frac{2}{r},1\},\infty)$, and Theorem \ref{Tharea}, we conclude that
\begin{align*}
\lf\|g_\lambda^*(f)\r\|_{X}^\nu&=\lf\|\lf[g_\lambda^*(f)\r]^\nu\r\|_{X^{1/\nu}}
\lesssim\lf\|\lf[\ca^{(1)}\lf(F\r)\r]^\nu\r\|_{X^{1/\nu}}
+\sum_{m=0}^\infty2^{\frac{-\lambda nm\nu}{2}}\lf\|\lf[\ca^{(2^{m+1})}
\lf(F\r)\r]^\nu\r\|_{X^{1/\nu}}\\
&\lesssim\lf\|\ca^{(1)}\lf(F\r)\r\|_{X}^\nu+\sum_{m=0}^\infty2^{\frac{-\lambda nm\nu}{2}}\max\lf\{2^{\frac{mn\nu}{2}},2^\frac{mn\nu}{r}\r\}
\lf\|\ca^{(1)}\lf(F\r)\r\|_{X}^\nu\\
&\lesssim\|S(f)\|_{X}^\nu\sim\|f\|_{H_X(\rn)}^\nu,
\end{align*}
where, in the penultimate step,
we used the fact that $\ca^{(1)}\lf(F\r)=S(f)$.
This finishes the proof of the necessity and hence of Theorem \ref{Thgx}.
\end{proof}

\subsection{Characterization by the Littlewood--Paley $g$-function\label{s4.3}}

We have the following Littlewood--Paley $g$-function characterization of $H_X(\rn)$.
\begin{theorem}\label{Thgf}
Assume that $X$ is a ball quasi-Banach function space satisfying Assumption \ref{as1}
with $0<\theta< s\le1$ and Assumption \ref{as2} with the same $s$ as in Assumption \ref{as1}.
Assume that there exists a positive
constant $C$ such that, for any $\{f_j\}_{j=1}^\infty\subset\mathscr M(\rn)$,
\begin{equation}\label{as1-3}
\lf\|\lf\{\sum_{j=1}^\infty\lf[\cm^{(\theta)}(f_j)\r]^s\r\}^\frac1s\r\|_{X^{s/2}}
\le C\lf\|\lf\{\sum_{j=1}^\infty|f_j|^s\r\}^\frac1s\r\|_{X^{s/2}}.
\end{equation}
Then $f\in H_X(\rn)$ if and only if $f\in\cs'(\rn)$, $f$ vanishes weakly at infinity and $\|g(f)\|_X<\infty$.
Moreover, for any $f\in H_X(\rn)$,
$$
\|f\|_{H_X(\rn)}\sim\|g(f)\|_X,
$$
where the positive equivalence constants are independent of $f$.
\end{theorem}

The following pointwise estimate is a slight variant of \cite[(2.66)]{U}, which
is just \cite[Lemma 3.21]{WYYZ}.

\begin{lemma}\label{Le67}
Let $\phi$ be a Schwartz function and assume that, for any $x\in\rn\setminus\{\vec 0_n\}$,
there exists a $j\in \zz$ such that $\widehat\phi(2^{j}x)\not=0$.
Then, for any given $N_0\in\nn$ and $\gamma\in(0,\infty)$,
there exists a positive constant $C_{(N_0,\gamma,\phi)}$, depending
only on $n$, $N_0$, $\gamma$ and $\phi$,
such that,
for any $s\in[1,2]$, $a\in(0,N_0]$, $l\in\zz$,
$f\in\cs'(\rn)$ and $x\in\rn$,
\begin{equation*}
\lf[\lf(\phi_{2^{-l}s}^\ast f\r)_a(x)\r]^\gamma\le C_{(N_0,\gamma,\phi
)}\sum_{k=0}^\infty2^{-kN_0\gamma}2^{(k+l)n}
\int_\rn\frac{|(\phi_{2^{-(k+l)}})_s\ast f(y)|^\gamma}{(1+2^l|x-y|)^{a\gamma}}\,dy.
\end{equation*}
\end{lemma}

\begin{proof}[Proof of Theorem \ref{Thgf}]
By an argument similar to that used in the proof of \cite[Theorem 2.10(i)]{WYY},
we obtain the necessity of this theorem.

Conversely, repeating the proof of \cite[Theorem 2.10(i)]{WYY} via replacing
\cite[Lemma 2.14]{WYY} used therein by Lemma \ref{Le67} here,
we complete the proof of the sufficiency
and hence of Theorem \ref{Thgf}.
\end{proof}

\section{Applications\label{s5}}
In this section, we apply Theorems \ref{Tharea}, \ref{Thgx} and \ref{Thgf}, respectively, to five concrete examples
of ball quasi-Banach function spaces, namely, Morrey spaces (see \S\ref{s5.1} below),
mixed-norm Lebesgue spaces (see \S\ref{s5.2} below),
variable Lebesgue spaces (see \S\ref{s5.3} below), weighted Lebesgue spaces (see \S\ref{s5.4} below)
and Orlicz-slice spaces (see \S\ref{s5.5} below). Observe that, among these five examples, only variable Lebesgue spaces are
quasi-Banach function spaces as in Remark \ref{bbf}(ii),
while the other four examples are ball quasi-Banach function spaces, which are not necessary to be
quasi-Banach function spaces.

\subsection{Morrey spaces\label{s5.1}}

Recall that, due to the applications in elliptic partial differential equations,
the Morrey space $M_r^p(\rn)$ with $0<r\le p<\infty$
was introduced by Morrey \cite{Mo} in 1938.
In recent decades,
there exists an increasing interest in applications
of Morrey spaces to various areas of analysis
such as partial differential equations, potential theory and
harmonic analysis (see, for instance, \cite{a15,a04,cf,JW,ysy}).

\begin{definition}
Let $0<r\le p<\infty$.
The \emph{Morrey space $M_r^p(\rn)$} is defined
to be the set of all measurable functions $f$ such that
$$
\|f\|_{M_r^p(\rn)}:=\sup_{B\in\BB}|B|^{1/p-1/r}\|f\|_{L^q(B)}<\infty,
$$
where $\BB$ is as in \eqref{Eqball} (the set of all balls of $\rn$).
\end{definition}

\begin{remark}
Observe that, as was pointed out in \cite[p.\,86]{SHYY}, $M_r^p(\rn)$
may not be a quasi-Banach function space, but it is a ball
quasi-Banach function space as in Definition \ref{Debqfs}.
\end{remark}

Let $0<r\le p<\infty$. From \cite[Theorem 2.4]{st} and \cite[Lemma 2.5]{tx},
it follows that $M_r^p(\rn)$ satisfies Assumption \ref{as1} for any
$\theta,\ s\in(0,\min\{1,r\})$ with $\theta<s$ (see also \cite{cf,H15}).
Applying  \cite[Lemma 5.7]{H15} and \cite[Theorem 4.1]{ST},
we can easily show that Assumption \ref{as2}
holds true for any given $s\in(0,r)$ and $q\in(r,\infty]$, and $X:=M_r^p(\rn)$.
Thus, all the assumptions of main theorems in Section \ref{s4} are satisfied.
Using Theorems
\ref{Tharea}, \ref{Thgf} and \ref{Thgx},
we obtain
the following characterizations of Morrey--Hardy space $HM_r^p(\rn)$, respectively, in terms of the
Lusin area function, the Littlewood--Paley $g$-function and the Littlewood--Paley $g_\lambda^*$-function.

\begin{theorem}\label{Thsm}
Let $p,\ r\in(0,\infty)$ with $r\leq p$.
Then $f\in HM_r^p(\rn)$ if and only if either of the following two items holds true:
\begin{itemize}
\item[{\rm(i)}]$f\in\cs'(\rn)$, $f$ vanishes weakly at infinity and
$S(f)\in M_r^p(\rn)$, where $S(f)$ is as in \eqref{deaf}.
\item[{\rm(ii)}] $f\in\cs'(\rn)$, $f$ vanishes weakly at infinity and
$g(f)\in M_r^p(\rn)$, where $g(f)$ is as in \eqref{degf}.
\end{itemize}
Moreover, for any $f\in HM_r^p(\rn)$,
$$
\|f\|_{HM_r^p(\rn)}\sim\|S(f)\|_{M_r^p(\rn)}\sim\|g(f)\|_{M_r^p(\rn)},
$$
where the positive equivalence constants are independent of $f$.
\end{theorem}
\begin{remark}
If $\varphi$ appearing in the definitions of $S(f)$ and $g(f)$ as in \eqref{deaf} and \eqref{degf} satisfies that
$\mathbf1_{B(\vec 0_n,4)\setminus B(\vec0_n,2)}\le\widehat\varphi\le\mathbf1_{B(\vec 0_n,8)\setminus B(\vec0_n,1)}$,
then,
in this case, Theorem \ref{Thoss} was obtained by Sawano et al. \cite[Theorem 3.21]{SHYY}.
\end{remark}

\begin{theorem}\label{ThM}
Let $p,\ r\in(0,\infty)$ with $r\leq p$, and $\lambda\in(\max\{1,2/r\},\infty)$.
Then $f\in HM_r^p(\rn)$ if and only if $f\in\cs'(\rn)$, $f$ vanishes weakly at infinity and $\|g_\lambda^\ast(f)\|_{M_r^p(\rn)}<\infty$.
Moreover, for any $f\in HM_r^p(\rn)$,
$$
\|f\|_{HM_r^p(\rn)}\sim\|g_\lambda^\ast(f)\|_{M_r^p(\rn)},
$$
where the positive equivalence constants are independent of $f$.
\end{theorem}

\begin{remark}
If $p\in(0,1]$, $r\in(0,p]$ and $\varphi$ appearing in the definition of $g_\lambda^*(f)$ as in \eqref{d22} satisfies that
$\mathbf1_{B(\vec 0_n,4)\setminus B(\vec0_n,2)}\le\widehat\varphi\le\mathbf1_{B(\vec 0_n,8)\setminus B(\vec0_n,1)}$
then,
in this case, Wang et al. \cite[Corollary 2.22(v)]{WYY} also obtained the same result as in Theorem \ref{ThM}.
When $p\in(1,\infty)$ and $r\in(0,p]$, to the best of our knowledge, the result of Theorem \ref{ThM} is new.
\end{remark}

\subsection{Mixed-norm Lebesgue spaces\label{s5.2}}

The mixed-norm Lebesgue space $L^{\vec{p}}(\rn)$
was studied by Benedek and Panzone \cite{BP} in 1961, which can be traced
back to Hormander \cite{H1}. Later on, in 1970, Lizorkin \cite{l70} further developed both the theory of
multipliers of Fourier integrals and estimates of convolutions
in the mixed-norm Lebesgue spaces.
Particularly,
in order to meet the requirements arising in the study of the boundedness of operators,
partial differential equations and some other analysis fields, the real-variable theory of mixed-norm
function spaces, including mixed-norm Morrey spaces,
mixed-norm Hardy spaces, mixed-norm Besov
spaces and mixed-norm Triebel--Lizorkin spaces, has rapidly been developed
in recent years (see, for instance,
\cite{cgn17bs,gjn17,tn,hy,HLYY,hlyy19}).

\begin{definition}\label{mix}
Let $\vec{p}:=(p_1,\ldots,p_n)\in(0,\infty]^n$.
The \emph{mixed-norm Lebesgue space $L^{\vec{p}}(\rn)$} is defined
to be the set of all measurable functions $f$ such that
$$
\|f\|_{L^{\vec{p}}(\rn)}:=\lf\{\int_{\rr}\cdots\lf[\int_{\rr}|f(x_1,\ldots,x_n)|^{p_1}\,dx_1\r]
^{\frac{p_2}{p_1}}\cdots\,dx_n\r\}^{\frac{1}{p_n}}<\infty
$$
with the usual modifications made when $p_i=\infty$ for some $i\in\{1,\ldots,n\}$.
\end{definition}

In this subsection, for any $\vec{p}:=(p_1,\ldots,p_n)\in(0,\infty)^n$, we always let
$p_-:= \min\{p_1, \ldots , p_n\}$ and  $p_+ := \max\{p_1, \ldots , p_n\}$.

Let $\vec{p}\in(0,\infty)^n$. Then
$L^{\vec{p}}(\rn)$ satisfies
Assumption \ref{as1} for any
$\theta,\ s\in(0,\min\{1,p_-\})$ with $\theta<s$ (see \cite[Lemma 3.7]{HLYY}).
Applying \cite[Lemma 3.5]{HLYY} and \cite[Theorem 1.a]{BP},
we can easily show that Assumption \ref{as2}
holds true for any given $s\in(0,p_-)$ and $q\in(p_+,\infty]$, and $X:=L^{\vec{p}}(\rn)$.
Thus, all the assumptions of main theorems in Sections \ref{s3} and \ref{s4} are satisfied.
Using Theorems
\ref{Tharea}, \ref{Thgf} and \ref{Thgx},
we obtain
the following characterizations of the mixed Hardy space $H^{\vec{p}}(\rn)$, respectively, in terms of the
Lusin area function, the Littlewood--Paley $g$-function and the Littlewood--Paley $g_\lambda^*$-function.

\begin{theorem}\label{Thsmi}
Let $\vec p:=(p_1,\ldots,p_n)\in(0,\infty)^n$.
Then $f\in H^{\vec p}(\rn)$ if and only if
either of the following two items holds true:
\begin{itemize}
\item[{\rm(i)}]$f\in\cs'(\rn)$, $f$ vanishes weakly at infinity and
$S(f)\in L^{\vec p}(\rn)$, where $S(f)$ is as in \eqref{deaf}.
\item[{\rm(ii)}] $f\in\cs'(\rn)$, $f$ vanishes weakly at infinity and
$g(f)\in L^{\vec p}(\rn)$, where $g(f)$ is as in \eqref{degf}.
\end{itemize}
Moreover, for any $f\in H^{\vec p}(\rn)$,
$$
\|f\|_{H^{\vec p}(\rn)}\sim\|S(f)\|_{L^{\vec p}(\rn)}\sim\|g(f)\|_{L^{\vec p}(\rn)},
$$
where the positive equivalence constants are independent of $f$.
\end{theorem}
\begin{remark}\label{Re5.16}
Let $\vec p:=(p_1,\ldots,p_n)\in(0,\infty)^n$. If $\varphi$ appearing in
the definitions of $S(f)$ and $g(f)$ as in \eqref{deaf} and \eqref{degf} satisfies that
$\varphi\in\cs(\rn)$ is a radial function such that, for any multi-index
$\alpha\in\zz_+^n$ with $|\alpha|\leq\lfloor n(\frac1{\min\{p_1,\ldots,p_n\}}-1)\rfloor$,
$\int_\rn x^\alpha\varphi(x)\,dx=0$ and, for any $\xi\in\rn\setminus\{\vec0_n\}$,
$\sum_{k\in\zz}|\widehat\varphi(2^k\xi)|^2=1$, then, in this case,
Theorem \ref{Thsmi} was obtained by Huang et al. \cite[Theorems 4.1 and 4.2]{HLYY}
as a special case.
\end{remark}

\begin{theorem}\label{ThMix}
Let $\vec p:=(p_1,\ldots,p_n)\in(0,\infty)^n$. Let $\lambda\in(\max\{1,\frac{2}{\min\{p_1,\ldots,p_n\}}\},\infty)$.
Then $f\in H^{\vec p}(\rn)$ if and only if $f\in\cs'(\rn)$,
$f$ vanishes weakly at infinity and $\|g_\lambda^\ast(f)\|_{L^{\vec p}(\rn)}<\infty$.
Moreover, for any $f\in H^{\vec p}(\rn)$,
$$
\|f\|_{H^{\vec p}(\rn)}\sim\|g_\lambda^\ast(f)\|_{L^{\vec p}(\rn)},
$$
where the positive equivalence constants are independent of $f$.
\end{theorem}
\begin{remark}
Let $\vec p:=(p_1,\ldots,p_n)\in(0,\infty)^n$. If $\varphi$
appearing in the definition of $g_\lambda^*(f)$ as in \eqref{d22} is as in Remark \ref{Re5.16}, then,
in this case, Theorem \ref{ThMix} widens the range of
$\lambda\in(1+\frac{2}{\min\{2,p_1,\ldots,p_n\}},\infty)$ in \cite[Theorem 4.3]{HLYY}
into $\lambda\in(\max\{1,\frac{2}{\min\{p_1,\ldots,p_n\}}\},\infty)$.
\end{remark}
\subsection{Variable Lebesgue spaces\label{s5.3}}

Let $p(\cdot):\ \rn\to[0,\infty)$ be a measurable function. Then the \emph{variable Lebesgue space
$L^{p(\cdot)}(\rn)$} is defined to be the set of all measurable functions $f$ on $\rn$ such that
$$
\|f\|_{L^{p(\cdot)}(\rn)}:=\inf\lf\{\lambda\in(0,\infty):\ \int_\rn[|f(x)|/\lambda]^{p(x)}\,dx\le1\r\}<\infty.
$$
We refer the reader to \cite{N1,N2,KR,CUF,DHR} for more details on variable Lebesgue spaces.

For any measurable function $p(\cdot):\ \rn\to(0,\infty)$, in this subsection, we let
$$
 \widetilde{p_-}:=\underset{x\in\rn}{\essinf}\,p(x)\quad\text{and}\quad
 \widetilde p_+:=\underset{x\in\rn}{\esssup}\,p(x).
$$
If $0<\widetilde p_-\le \widetilde p_+<\infty$, then, similarly to the proof of \cite[Theorem 3.2.13]{DHHR},
we know that $L^{p(\cdot)}(\rn)$ is
a quasi-Banach function space and hence a ball quasi-Banach function space.

A measurable function $p(\cdot):\ \rn\to(0,\infty)$ is said to be \emph{globally
log-H\"older continuous} if there exists a $p_{\infty}\in\rr$ such that, for any
$x,\ y\in\rn$,
$$
|p(x)-p(y)|\lesssim\frac{1}{\log(e+1/|x-y|)}
$$
and
$$
|p(x)-p_\infty|\lesssim\frac{1}{\log(e+|x|)},
$$
where the positive equivalence constants are independent of $x$ and $y$.

Let $p(\cdot):\ \rn\to(0,\infty)$ be a globally
log-H\"older continuous function satisfying
$0<\widetilde p_-\le \widetilde p_+<\infty$.
Adamowicz et al. \cite{ahh} obtained the
boundedness of the Hardy--Littlewood maximal operator on variable Lebesgue spaces.
Using this and \cite[Theorem 3.1]{AJ}, we can readily prove that, for any
$\theta,\ s\in(0,\min\{1,\widetilde p_-\})$ with $\theta<s$, Assumption \ref{as1} is satisfied
(see also \cite{CUF,CUW}). Furthermore, from \cite[Lemma 2.16]{CUF},
we deduce that
Assumption \ref{as2} holds true for any given $s\in(0,\widetilde p_-)$ and
$q\in(\widetilde p_+,\infty]$,
and $X:=L^{p(\cdot)}(\rn)$.
Thus, all the assumptions of main theorems in Section \ref{s4} are satisfied.
Using Theorems
\ref{Tharea}, \ref{Thgf} and \ref{Thgx},
we obtain the following characterizations of variable Hardy space $H^{p(\cdot)}(\rn)$, respectively, in terms of the
Lusin area function, the Littlewood--Paley $g$-function and the Littlewood--Paley $g_\lambda^*$-function.

\begin{theorem}\label{Thsv}
Let $p(\cdot):\ \rn\to(0,\infty)$ be a globally
log-H\"older continuous function satisfying
$0<\widetilde p_-\leq \widetilde p_+<\infty$.
Then $f\in H^{\vec p}(\rn)$ if and only if
either of the following two items holds true:
\begin{itemize}
\item[{\rm(i)}]$f\in\cs'(\rn)$, $f$ vanishes weakly at infinity and
$S(f)\in L^{p(\cdot)}(\rn)$, where $S(f)$ is as in \eqref{deaf}.
\item[{\rm(ii)}] $f\in\cs'(\rn)$, $f$ vanishes weakly at infinity and
$g(f)\in L^{p(\cdot)}(\rn)$, where $g(f)$ is as in \eqref{degf}.
\end{itemize}
Moreover, for any $f\in H^{p(\cdot)}(\rn)$,
$$
\|f\|_{H^{p(\cdot)}(\rn)}\sim\|S(f)\|_{L^{p(\cdot)}(\rn)}\sim\|g(f)\|_{L^{p(\cdot)}(\rn)},
$$
where the positive equivalence constants are independent of $f$.
\end{theorem}
\begin{remark}\label{Rev1}
If $\varphi$ appearing in the definitions of $S(f)$ and $g(f)$ as in \eqref{deaf} and \eqref{degf} satisfies that
$\mathbf1_{B(\vec 0_n,4)\setminus B(\vec0_n,2)}\le\widehat\varphi\le\mathbf1_{B(\vec 0_n,8)\setminus B(\vec0_n,1)}$,
then,
in this case, Theorem \ref{Thsv} was obtained by \cite[Theorem 3.21]{SHYY} and
\cite[Corollary 2.22(vi)]{WYY}.
\end{remark}

\begin{theorem}\label{Thv}
Let $p(\cdot):\ \rn\to(0,\infty)$ be a globally
log-H\"older continuous function satisfying
$0<\widetilde p_-\leq \widetilde p_+<\infty$. Let $\lambda\in(\max\{1,\frac{2}{\widetilde p_-}\},\infty)$.
Then $f\in H^{p(\cdot)}(\rn)$ if and only if $f\in\cs'(\rn)$, $f$ vanishes weakly at infinity and $\|g_\lambda^\ast(f)\|_{L^{p(\cdot)}(\rn)}<\infty$.
Moreover, for any $f\in H^{p(\cdot)}(\rn)$,
$$
\|f\|_{H^{p(\cdot)}(\rn)}\sim\|g_\lambda^\ast(f)\|_{L^{p(\cdot)}(\rn)},
$$
where the positive equivalence constants are independent of $f$.
\end{theorem}
\begin{remark}
Let $p(\cdot):\ \rn\to(0,\infty)$ be a globally
log-H\"older continuous function satisfying
$0<\widetilde p_-\leq \widetilde p_+<\infty$. Let
$\lambda\in(\max\{1,\frac{2}{\widetilde p_-}\},\infty)$.
If $\varphi$ appearing in the definition of $g_\lambda^*(f)$ as in \eqref{d22} satisfies that
$\mathbf1_{B(\vec 0_n,4)\setminus B(\vec0_n,2)}\le\widehat\varphi\le\mathbf1_{B(\vec 0_n,8)\setminus B(\vec0_n,1)}$,
then,
in this case, Theorem \ref{Thos} widens the range of
$\lambda\in(\max\{\frac{2}{\min\{1,\widetilde p_-\}},1-\frac{2}
{\max\{1,\widetilde p_+\}}+\frac{2}{\min\{1,\widetilde p_-\}}\},\infty)$ in \cite[Corollary 2.22(vi)]{WYY}
into $\lambda\in(\max\{1,\frac{2}{\widetilde p_-}\},\infty)$.
\end{remark}

\subsection{Weighted Lebesgue spaces\label{s5.4}}

If $X$ is a weighted Lebesgue space $L_\omega^p(\rn)$ with $p\in(0,\infty)$ and
$\omega\in A_\infty(\rn)$, then $H_X(\rn)$ is just the
weighted Hardy space $H_\omega^p(\rn)$, which was studied in \cite{Bui,GC2,LHY,ccyy,STo,yy11,YLK}.

It is worth pointing out that a weighted Lebesgue space with an $A_\infty(\rn)$-weight may
not be a Banach function space; see \cite[Section 7.1]{SHYY}.
From \cite[Theorem 2.7.4]{DHHR},
we deduce that, when $p\in(1,\infty)$ and $\omega\in A_\infty(\rn)$,
$[L_\omega^p(\rn)]'=L^{p'}_{\omega^{1-p'}}(\rn)$, where $[L_\omega^p(\rn)]'$ is the associated space of $L_\omega^p(\rn)$
as in \eqref{asso} with $X:=L^p_\omega(\rn)$. For any given $\omega\in A_\infty(\rn)$, let
\begin{equation}\label{eq-qw}
q_\omega:=\inf\lf\{q\in[1,\infty):\ \omega\in A_q(\rn)\r\}.
\end{equation}
Let $p\in(0,\infty)$ and $\omega\in A_\infty(\rn)$. Obviously, $\omega\in A_{p/\theta}(\rn)$ holds true for any $\theta\in(0,p/q_\omega)$.
From this and \cite[Theorem 3.1(b)]{AJ}, we deduce that $L^p_\omega(\rn)$ satisfies Assumption \ref{as1}
for any $\theta,\ s\in(0,\min\{1,p/q_\omega\})$ with $\theta<s$.
Moreover, by \cite[Proposition 7.1.5(4)]{G1},
we know that, for any $s\in(0,p/q_\omega)$, $\omega^{1-(p/s)'}\in A_{(p/s)'}(\rn)$
holds true. By this and \cite[Corollary 7.2.6]{G1}, we conclude that there exists
a $q\in(0,\infty)$ such that $\omega^{1-(p/s)'}\in A_{(p/s)'/( q/s)'}(\rn)$.
Then, from \cite[Theorem 7.1.9(b)]{G1} and the fact that
$L_{\omega^{1-(p/s)'}}^{(p/s)'/(q/s)'}(\rn)=[L_{\omega^{1-(p/s)'}}^{(p/s)'}(\rn)]^{1/(q/s)'}$,
we deduce that $\cm^{(q/s)'}$ is bounded on $L_{\omega^{1-(p/s)'}}^{(p/s)'}(\rn)$,
which shows that, for any given $s\in(0,p/q_\omega)$, there exists a
$q\in(0,\infty)$ such that Assumption \ref{as2} holds true.
Thus, all the assumptions of main theorems in Section \ref{s4} are satisfied.
Using Theorems
\ref{Tharea}, \ref{Thgf} and \ref{Thgx},
we immediately obtain the following characterizations of $H_\omega^p(\rn)$ by means of the
Lusin area function, the Littlewood--Paley $g$-function and
the Littlewood--Paley $g_\lambda^*$-function.

\begin{theorem}\label{Thwos}
Let $p\in(0,\infty)$ and $\omega\in A_\infty(\rn)$.
Then $f\in H^p_\omega(\rn)$ if and only if
either of the following two items holds true:
\begin{itemize}
\item[{\rm(i)}]$f\in\cs'(\rn)$, $f$ vanishes weakly at infinity and
$S(f)\in L^p_\omega(\rn)$, where $S(f)$ is as in \eqref{deaf}.
\item[{\rm(ii)}] $f\in\cs'(\rn)$, $f$ vanishes weakly at infinity and
$g(f)\in L^p_\omega(\rn)$, where $g(f)$ is as in \eqref{degf}.
\end{itemize}
Moreover, for any $f\in H^p_\omega(\rn)$,
$$
\|f\|_{H^p_\omega(\rn)}\sim\|S(f)\|_{L^p_\omega(\rn)}\sim\|g(f)\|_{L^p_\omega(\rn)},
$$
where the positive equivalence constants are independent of $f$.
\end{theorem}
\begin{remark}\label{Re5.40}
Let $p\in(0,1]$ and $\omega\in A_\infty(\rn)$. Assume that $\varphi\in\cs(\rn)$
appearing in the definitions of $S(f)$ and $g(f)$ as in \eqref{deaf} and \eqref{degf}
is a radial function supported in the unit ball $B(\vec0_n,1)$ satisfying that, for any multi-index
$\alpha\in\zz_+^n$ with $|\alpha|\leq\lfloor n({q_\omega}/{p}-1)\rfloor$,
$\int_\rn x^\alpha\varphi(x)\,dx=0$ and, for any $\xi\in\rn\setminus\{\vec 0_n\}$,
$\int_0^\infty|\widehat\varphi(t\xi)|^2\,\frac{dt}{t}=1$.
In this case, Thereom \ref{Thwos} was obtained by Liang et al. \cite[Theorem 4.13]{HYY}
as a special case.
\end{remark}

\begin{theorem}\label{Thwo}
Let $p\in(0,\infty)$, $\omega\in A_\infty(\rn)$ and
$\lambda\in(\max\{1,2q_\omega/{p}\},\infty)$, where $q_\omega$ is as in \eqref{eq-qw}.
Then $f\in H^p_\omega(\rn)$ if and only if $f\in\cs'(\rn)$,
$f$ vanishes weakly at infinity and $\|g_\lambda^*(f)\|_{L^p_\omega(\rn)}<\infty$.
Moreover, for any $f\in H^p_\omega(\rn)$,
$$
\|f\|_{H^p_\omega(\rn)}\sim\|g_\lambda^\ast(f)\|_{L^p_\omega(\rn)},
$$
where the positive equivalence constants are independent of $f$.
\end{theorem}
\begin{remark}
Let $p\in(0,1]$ and $\varphi$ appearing in the definition of $g_\lambda^*(f)$ in \eqref{d22} be as in
Remark \ref{Re5.40}. In this case, Thereom \ref{Thwo} was obtained by Liang et al. \cite[Theorem 4.8]{LHY}
as a special case.
\end{remark}

\subsection{Orlicz-slice spaces\label{s5.5}}

First, we recall the notions of both Orlicz functions and Orlicz spaces (see, for instance, \cite{RR}).

\begin{definition}\label{or}
A function $\Phi:\ [0,\infty)\ \to\ [0,\infty)$ is called an \emph{Orlicz function} if it is
non-decreasing and satisfies $\Phi(0)= 0$, $\Phi(t)>0$ whenever $t\in(0,\infty)$, and $\lim_{t\to\infty}\Phi(t)=\infty$.
\end{definition}

An Orlicz function $\Phi$ as in Definition \ref{or} is said to be
of \emph{lower} (resp., \emph{upper}) \emph{type} $p$ with
$p\in(-\infty,\infty)$ if
there exists a positive constant $C_{(p)}$, depending on $p$, such that, for any $t\in[0,\infty)$
and $s\in(0,1)$ [resp., $s\in [1,\infty)$],
\begin{equation*}
\Phi(st)\le C_{(p)}s^p \Phi(t).
\end{equation*}
A function $\Phi:\ [0,\infty)\ \to\ [0,\infty)$ is said to be of
 \emph{positive lower} (resp., \emph{upper}) \emph{type} if it is of lower
 (resp., upper) type $p$ for some $p\in(0,\infty)$.
\begin{definition}\label{fine}
Let $\Phi$ be an Orlicz function with positive lower type $p_{\Phi}^-$ and positive upper type $p_{\Phi}^+$.
The \emph{Orlicz space $L^\Phi(\rn)$} is defined
to be the set of all measurable functions $f$ such that
 $$\|f\|_{L^\Phi(\rn)}:=\inf\lf\{\lambda\in(0,\infty):\ \int_{\rn}\Phi\lf(\frac{|f(x)|}{\lambda}\r)\,dx\le1\r\}<\infty.$$
\end{definition}
Now we recall the notion of Orlicz-slice spaces.
\begin{definition}\label{so}
Let $t,\ r\in(0,\infty)$ and $\Phi$ be an Orlicz function with positive lower type $p_{\Phi}^-$ and
positive upper type $p_{\Phi}^+$. The \emph{Orlicz-slice space} $(E_\Phi^r)_t(\rn)$
is defined to be the set of all measurable functions $f$
such that
$$
\|f\|_{(E_\Phi^r)_t(\rn)}
:=\lf\{\int_{\rn}\lf[\frac{\|f\mathbf{1}_{B(x,t)}\|_{L^\Phi(\rn)}}
{\|\mathbf{1}_{B(x,t)}\|_{L^\Phi(\rn)}}\r]^r\,dx\r\}^{\frac{1}{r}}<\infty.
$$
\end{definition}

\begin{remark}
By \cite[Lemma 2.28]{ZYYW}, we know that the Orlicz-slice space $(E_\Phi^r)_t(\rn)$
is a ball quasi-Banach function space, but it may not be a quasi-Banach function space
(see, for instance, \cite[Remark 7.41(i)]{ZWYY})
\end{remark}

The Orlicz-slice space
was introduced by Zhang et al. \cite{ZYYW}, which is a generalization of the slice spaces
proposed by Auscher and Mourgoglou \cite{AM2014} and  Auscher and  Prisuelos-Arribas \cite{APA}.
Let $t,\ r\in(0,\infty)$ and $\Phi$ be an Orlicz function with positive lower type $p_{\Phi}^-$ and
positive upper type $p_{\Phi}^+$.
Then $(E_\Phi^r)_t(\rn)$ satisfies Assumption \ref{as1} for any
$\theta,\ s\in(0,\min\{1,r,p_{\Phi}^-\})$ with $\theta<s$
(see \cite[Lemma 4.3]{ZYYW}). Furthermore, from \cite[Lemmas 4.4]{ZYYW},
we deduce that
Assumption \ref{as2} holds true for any given $s\in(0,\min\{r,p_{\Phi}^-\})$ and $q\in(\max\{r,p_{\Phi}^+\},\infty)$,
and $X:=(E_\Phi^r)_t(\rn)$.
Thus, all the assumptions of main theorems in Sections \ref{s3} and \ref{s4} are satisfied.
Using Theorems
\ref{Tharea}, \ref{Thgf} and \ref{Thgx},
we immediately obtain the following characterizations of $(HE_\Phi^r)_t(\rn)$ in terms of the
Lusin area function, the Littlewood--Paley $g$-function and
the Littlewood--Paley $g_\lambda^*$-function.

\begin{theorem}\label{Thoss}
Let $t\in(0,\infty)$, $r,\ p_{\Phi}^-,\ p_{\Phi}^+\in(0,\infty)$ and $\Phi$ be an Orlicz function
with positive lower type $p_{\Phi}^-$ and positive upper type $p_{\Phi}^+$.
Then $f\in (HE_\Phi^r)_t(\rn)$ if and only if
either of the following two items holds true:
\begin{itemize}
\item[{\rm(i)}]$f\in\cs'(\rn)$, $f$ vanishes weakly at infinity and
$S(f)\in (E_\Phi^r)_t(\rn)$, where $S(f)$ is as in \eqref{deaf}.
\item[{\rm(ii)}] $f\in\cs'(\rn)$, $f$ vanishes weakly at infinity and
$g(f)\in (E_\Phi^r)_t(\rn)$, where $g(f)$ is as in \eqref{degf}.
\end{itemize}
Moreover, for any $f\in (HE_\Phi^q)_t(\rn)$,
$$
\|f\|_{(HE_\Phi^r)_t(\rn)}\sim\|S(f)\|_{(E_\Phi^r)_t(\rn)}\sim\|g(f)\|_{(E_\Phi^r)_t(\rn)},
$$
where the positive equivalence constants are independent of $f$.
\end{theorem}
\begin{remark}
Let $\varphi$ appearing in the definitions of $S(f)$ and $g(f)$ as in \eqref{deaf} and \eqref{degf} satisfy that
$\mathbf1_{B(\vec 0_n,4)\setminus B(\vec0_n,2)}\le\widehat\varphi\le\mathbf1_{B(\vec 0_n,8)\setminus B(\vec0_n,1)}$.
In this case, Theorem \ref{Thoss} was obtained by Zhang et al. \cite[Theorem 3.17]{ZYYW}.
\end{remark}

\begin{theorem}\label{Thos}
Let $t\in(0,\infty)$, $r,\ p_{\Phi}^-,\ p_{\Phi}^+\in(0,\infty)$ and $\Phi$ be an Orlicz function
with positive lower type $p_{\Phi}^-$ and positive upper type $p_{\Phi}^+$.
Let $\lambda\in(\max\{1,\frac{2}{\min\{p_\Phi^-,r\}}\},\infty)$.
Then $f\in (HE_\Phi^r)_t(\rn)$ if and only if $f\in\cs'(\rn)$, $f$ vanishes weakly at infinity
and $\|g_\lambda^*(f)\|_{(E_\Phi^r)_t(\rn)}<\infty$, where $g_\lambda^*(f)$ is as in \eqref{d22}.
Moreover, for any $f\in (HE_\Phi^q)_t(\rn)$,
$$
\|f\|_{(HE_\Phi^r)_t(\rn)}\sim\|g_\lambda^\ast(f)\|_{(E_\Phi^r)_t(\rn)},
$$
where the positive equivalence constants are independent of $f$.
\end{theorem}
\begin{remark}
Let $r,\ p_{\Phi}^-,\ p_{\Phi}^+\in(0,\infty)$ and $\Phi$ be an Orlicz function
with positive lower type $p_{\Phi}^-$ and positive upper type $p_{\Phi}^+$.
Let $\varphi$ appearing in the definition of $g_\lambda^*(f)$ as in \eqref{d22} satisfy that
$$\mathbf1_{B(\vec 0_n,4)\setminus B(\vec0_n,2)}\le\widehat\varphi\le\mathbf1_{B(\vec 0_n,8)\setminus B(\vec0_n,1)}.
$$
In this case, Theorem \ref{Thos} widens the range of
$\lambda\in(\max\{\frac{2}{\min\{1,p_\Phi^-,r\}},1-\frac{2}{\max\{1,p_\Phi^+,r\}}+
\frac{2}{\min\{1,p_\Phi^-,r\}}\},\infty)$ in \cite[Theorem 3.19]{ZYYW}
into $\lambda\in(\max\{1,\frac{2}{\min\{p_\Phi^-,r\}}\},\infty)$.
\end{remark}



\bigskip

\noindent Der-Chen Chang

\medskip

\noindent Department of Mathematics,
Georgetown University, Washington D. C. 20057, USA

\noindent Graduate Institute of Business Administration, College of Management,
Fu Jen Catholic University, New Taipei City 242, Taiwan, ROC

\smallskip

\noindent{\it E-mail:} \texttt{chang@georgetown.edu}

\bigskip

\noindent Songbai Wang

\smallskip

\noindent College of Mathematics and Statistics, Hubei Normal University,
Huangshi 435002, People's Republic of China

\smallskip

\noindent{\it E-mail}: \texttt{haiyansongbai@163.com}

\bigskip

\noindent Dachun Yang (Corresponding author) and Yangyang Zhang (Corresponding author)

\smallskip

\noindent  Laboratory of Mathematics and Complex Systems
(Ministry of Education of China),
School of Mathematical Sciences, Beijing Normal University,
Beijing 100875, People's Republic of China

\smallskip

\noindent {\it E-mails}: \texttt{dcyang@bnu.edu.cn} (D. Yang)

\noindent\phantom{{\it E-mails:}} \texttt{yangyzhang@mail.bnu.edu.cn} (Y. Zhang)

\end{document}